\documentclass[11pt, one side,emlines]{amsart}
\usepackage[a4paper,margin=1in]{geometry}
\usepackage{fancyhdr}
\usepackage{enumerate}
\usepackage{amssymb,latexsym,xy,eucal,mathrsfs}
\textwidth=18cm \textheight=27cm \theoremstyle{plain}
\newtheorem{theorem}{Theorem}[section]
\newtheorem{lemma}[theorem]{Lemma}

\newtheorem{proposition}[theorem]{Proposition}

\newtheorem{definition}[theorem]{Definition}

\usepackage{enumitem} 
\numberwithin{equation}{section} 
\begin{document}
\baselineskip 12 truept
\title{On graphic elementary lifts of graphic matroids }
%\date{}
\author{Ganesh Mundhe$^1$}
\address{\rm 1. Army Institute of Technology, Pune-411015, INDIA }
\email{ganumundhe@gmail.com}
\author{Y. M. Borse$^2$}
 \address{\rm 2. Department of Mathematics, Savitribai Phule Pune University, Pune-411007, INDIA}
\email{ymborse11@gmail.com}
 \author{K. V. Dalvi$^3$}
   \address{\rm 3. Government College of Engineering, Pune-411005, INDIA}
 \email{kvd.maths@coep.ac.in}
 \maketitle{}
 %\noindent
\begin{abstract}  Zaslavsky introduced the concept of lifted-graphic matroid. For binary matroids, a binary elementary lift can be defined in terms of the splitting operation.  In this paper,  we give a method to get a forbidden-minor characterization for the class of graphic matroids whose all lifted-graphic matroids are also graphic using the splitting operation.  
\end{abstract}
\section{Introduction} 
For undefined notions and terminology, we refer to Oxley \cite{oxley2006matroid}. A matroid $M$ is \textit{quotient} of a matroid $N$ if there is a matroid $Q$ such that, for some $X\subset E(Q)$, $N=Q\backslash X$ and $M=Q/X$.  If $|X|=1$, then $M$ is \textit{an elementary quotient} of $N$. A matroid $N$ is a \textit{lift} of $M$ if $M$ is a quotient of $N$. If $M$ is an elementary quotient of $N$, then $N$ is \textit{an elementary lift} of $M$.  A matroid $N$ is a lifted-graphic matroid if there is a matroid $Q$ 
with $E(Q)=E(N) \cup e$ such that $ Q\backslash e=N$ and $Q/e$ is graphic.

The concept of lifted-graphic matroid was introduced by Zaslavsky \cite{zas}.  Lifted-graphic matroids play an important role in the matroid minors project of Geelen, Gerards and Whittle \cite{j1, j2}. Lifted-graphic matroids are  studied in \cite{cw, cg,  dfd, fd, zas}. 
 This class is minor-closed.  In \cite{cg}, it is proved that there exist infinitely many pairwise non-isomorphic excluded minors for the class of lifted-graphic matroids.   Frank and Mayhew \cite{fd} proved that, for a positive integer $r$,  there are only a finite number of excluded minors of rank $r$ for the class of lifted-graphic matroids. 
   Very less is known about the forbidden-minor characterizations of lifts. Here, we consider only \textit{binary matroids}. We study a quotient of lifted-graphic matroids and also give a way to represent a lifted-graphic binary matroid over $GF(2)$ using its quotient.  We relate an elementary binary lift of a binary matroid with the splitting operation for binary matroids.  Using the splitting operation, we obtain a forbidden-minor characterization for the class of graphic matroids whose all elementary lifts are also graphic.

Fleischner \cite{fleischner1990eulerian} introduced the splitting operation with respect to a pair of edges of a graph and  he characterized  Eulerian graphs and gave an algorithm to find all Eulerian trails in an Eulerian graph using this operation.  Raghunathan et al. \cite{raghunathan1998splitting}  extended this operation to binary matroids.  They defined the splitting operation for binary matroids   with respect to a pair of elements and used it to characterize the binary Eulerian matroids. Later on, Shikare et al. \cite{shikare2011generalized} generalized this operation by defining  for a binary matroid  with respect to a general set instead of a pair as follows. 

\begin{definition}\label{c11} \cite{shikare2011generalized} Let $M$ be a binary matroid with standard matrix representation $A$ over the field $GF(2)$ and let $T$ be a subset of $E(M)$.  Let $A_T$ be the matrix  obtained by adjoining one extra row to the matrix $A$ whose entries are 1 in the columns labeled by the elements of the set $T$ and zero otherwise.  The vector matroid of the matrix $A_T,$ denoted by $M_T,$  is called as the splitting matroid of  $M$ with respect to $T,$ and the  transition from $M$ to $M_{T}$ is called as the {\it splitting operation} with respect to $T.$  
\end{definition}

Let $M$ be a binary matroid. From the above definition, it is clear that  the splitting matroid $M_T$ of $M$ is binary.  In the second section, we prove that a splitting matroid of  $M$  is an elementary lift of $M$ and conversely, every binary elementary lift of $M$ is a splitting matroid of $M.$  Therefore an elementary binary lift of a given binary matroid can be obtained by the splitting operation. Hence we use the term  a \textit{splitting matroid} of $M$ for a binary elementary lift of $M$.

Bases, cocircuits and connectivity of splitting matroids are studied in  \cite{bm2, mills,  raghunathan1998splitting, shikareazadi}.

In general, the splitting operation does not preserve the   graphicness and cographicness properties of a given matroid (see  \cite{borse2014excluded, shikare2010excluded}).  Shikare and Waphare \cite{shikare2010excluded}  characterized the graphic matroids $M$ such that $M_T$ is graphic for any $T \subseteq E(M)$ with $|T|=2.$  They proved that there are four forbidden minors for this class. However, the authors in \cite{bm1} observed that one of these four minors is redundant and they restated the result of Shikare and Waphare \cite{shikare2010excluded}    as follows.

\begin{theorem} \cite{bm1} \label{c2bm} Let $M$ be a graphic matroid.  Then $M_T$ is graphic for any $T \subseteq E(M)$ with $|T|=2$ if and only if $M$ has no minor isomorphic to any of the circuit matroids $M(G_1)$, $M(G_2)$ and $M(G_3),$  where $G_1$, $G_2$ and $ G_3$ are the graphs as shown in Figure 1.
\end{theorem}
\begin{center}
	% This is a LaTeX picture output by TeXCAD.
	% File name: [fig 1.pic].
	% Version of TeXCAD: 4.3
	% Reference / build: 30-Jun-2012 (rev. 105)
	% For new versions, check: http://texcad.sf.net/
	% Options on the following lines.
	%\grade{\on}
	%\emlines{\off}
	%\epic{\off}
	%\beziermacro{\on}
	%\reduce{\on}
	%\snapping{\off}
	%\pvinsert{% Your \input, \def, etc. here}
	%\quality{8.000}
	%\graddiff{0.005}
	%\snapasp{1}
	%\zoom{4.0000}
	\unitlength .7mm % = 1.992pt
	\linethickness{0.4pt}
	\ifx\plotpoint\undefined\newsavebox{\plotpoint}\fi % GNUPLOT compatibility
	\begin{picture}(68.145,44.87)(0,0)
	\put(.81,19.37){\circle*{1.33}}
	\put(15.81,19.37){\circle*{1.33}}
	\put(15.81,32.37){\circle*{1.33}}
	\put(25.81,32.37){\circle*{1.33}}
	\put(25.81,19.37){\circle*{1.33}}
	\put(41.81,19.37){\circle*{1.33}}
	\put(41.81,32.37){\circle*{1.33}}
	\put(33.81,43.37){\circle*{1.33}}
	\put(51.14,32.37){\circle*{1.33}}
	\put(67.48,32.37){\circle*{1.33}}
	\put(67.48,19.37){\circle*{1.33}}
	\put(51.14,19.37){\circle*{1.33}}
	\put(59.14,43.37){\circle*{1.33}}
	\put(15.81,32.37){\line(0,-1){13}}
	\put(15.81,19.37){\line(-1,0){15}}
	\put(25.81,32.37){\line(1,0){16}}
	\put(41.81,32.37){\line(0,-1){13}}
	\put(41.81,19.37){\line(-1,0){16}}
	\put(25.81,19.37){\line(0,1){13}}
	\put(25.81,32.37){\line(3,4){8.33}}
	\put(33.81,43.37){\line(3,-4){8.33}}
	\put(51.14,32.37){\line(1,0){16.33}}
	\put(67.48,32.37){\line(0,-1){13}}
	\put(67.48,19.37){\line(-1,0){16.33}}
	\put(51.14,19.37){\line(0,1){13}}
	\put(51.14,32.37){\line(3,4){8.33}}
	\put(59.14,43.37){\line(3,-4){8.33}}
	\put(67.48,32.37){\line(-5,-4){16.33}}
	\put(67.48,19.37){\line(-5,4){16.33}}
	\put(59.14,43.37){\line(-1,-3){8}}
	\put(59.14,43.37){\line(1,-3){8}}
	\put(33.81,43.37){\line(-1,-3){8}}
	\qbezier(.81,19.37)(1.31,20)(.81,19.12)
	\qbezier(.81,19.12)(.69,20.12)(1.06,19.12)
	\qbezier(1.06,19.12)(.19,20.37)(.81,19.62)
	\put(.31,31.87){\circle*{1.12}}
	\put(.31,32.37){\line(1,0){16}}
	%\emline(16.31,32.37)(15.56,32.62)
	\multiput(16.31,32.37)(-.125,.041667){6}{\line(-1,0){.125}}
	%\end
	\put(15.56,32.62){\line(1,0){.25}}
	%\emline(15.56,19.62)(.06,32.37)
	\multiput(15.56,19.62)(-.058490566,.0481132075){265}{\line(-1,0){.058490566}}
	%\end
	\put(.31,32.37){\line(0,-1){12.75}}
	\qbezier(.56,31.87)(8.06,44.87)(15.56,32.87)
	%\emline(15.56,32.62)(.56,19.62)
	\multiput(15.56,32.62)(-.0555555556,-.0481481481){270}{\line(-1,0){.0555555556}}
	%\end
	\qbezier(.56,19.62)(7.06,7.25)(15.56,19.37)
	\qbezier(25.81,19.62)(32.31,7.62)(41.81,19.62)
	\put(7.31,15.05){\makebox(0,0)[cc]{}}
	\put(32.81,15.55){\makebox(0,0)[cc]{}}
	\put(43.31,24.12){\makebox(0,0)[cc]{}}
	\put(63.56,39.37){\makebox(0,0)[cc]{}}
	\put(55.31,40.37){\makebox(0,0)[cc]{}}
	\put(56.06,36.12){\makebox(0,0)[cc]{}}
	\put(60.56,36.12){\makebox(0,0)[cc]{}}
	\put(7.75,40.12){\makebox(0,0)[cc]{}}
	\put(33.75,.75){\makebox(0,0)[cc]{Figure 1}}
	\put(9.25,8.25){\makebox(0,0)[cc]{$G_1$}}
	\put(35,8.25){\makebox(0,0)[cc]{$G_2$}}
	\put(59.75,8.25){\makebox(0,0)[cc]{$G_3$}}
	\put(8,41.2){\makebox(0,0)[cc]{}}
	\put(8,16){\makebox(0,0)[cc]{}}
	\put(53.5,38.75){\makebox(0,0)[cc]{}}
	\put(59.75,16){\makebox(0,0)[cc]{}}
	\end{picture}
	\end{center}

In this paper, we give a method to obtain forbidden-minor characterization for the class of graphic matroids whose all binary elementary lifts are graphic. We prove  every binary elementary lift of a binary matroid $M$  is isomorphic to a splitting matroid $M_T$ for some $T \subseteq E(M)$.  First, we characterize the class of graphic matroids which yield graphic matroids under the splitting operation with respect to a \textit{set of three elements}.  Then we generalize the results for  set of any size.

 Given a graph $G$, let  $\tilde{G}$ be a member of the collection of graphs obtained by adding an edge $e$ to $G$  such that at least one end vertex of $e$ belongs to the vertex set $V(G)$.

 The following is the main result of the paper. 
\begin{theorem} \label{c2grspl3}
	Let $M$ be a graphic matroid. Then the splitting matroid $M_T$ is graphic for any $T\subseteq E(M)$ with $|T|=3$ if and only if  $M$  does not contain a minor isomorphic to any of the circuit matroids  $M(\tilde{G_4}),  M(\tilde{G_5}) $ and $M(G_k)$ for $k=6,7,8,9,$ where  $G_4, G_5, \dots, G_{9}$ are the graphs as shown in Figure 2.
\end{theorem} 

\begin{center} 
% This is a LaTeX picture output by TeXCAD.
% File name: [fig 2.pic].
% Version of TeXCAD: 4.3
% Reference / build: 30-Jun-2012 (rev. 105)
% For new versions, check: http://texcad.sf.net/
% Options on the following lines.
%\grade{\on}
%\emlines{\off}
%\epic{\off}
%\beziermacro{\on}
%\reduce{\on}
%\snapping{\off}
%\pvinsert{% Your \input, \def, etc. here}
%\quality{8.000}
%\graddiff{0.005}
%\snapasp{1}
%\zoom{4.0000}
\unitlength .7mm % = 1.992pt
\linethickness{0.4pt}
\ifx\plotpoint\undefined\newsavebox{\plotpoint}\fi % GNUPLOT compatibility
\begin{picture}(160.125,42.875)(0,0)
\put(107.81,17.12){\circle*{1.33}}
\put(122.81,17.12){\circle*{1.33}}
\put(122.81,33.12){\circle*{1.33}}
\put(107.81,33.12){\circle*{1.33}}
\put(133.31,25.37){\circle*{1.33}}
\put(95.56,25.62){\circle*{1.33}}
\put(122.81,17.12){\line(-1,0){15}}
\qbezier(107.81,17.12)(108.31,17.75)(107.81,16.87)
\qbezier(107.81,16.87)(107.69,17.87)(108.06,16.87)
\qbezier(108.06,16.87)(107.19,18.12)(107.81,17.37)
\put(80,-1){\makebox(0,0)[cc]{Figure 2}}
\put(115.75,6.25){\makebox(0,0)[cc]{$G_{8}$}}
%\emline(95.5,25.75)(107.75,33)
\multiput(95.5,25.75)(.081125828,.048013245){151}{\line(1,0){.081125828}}
%\end
%\emline(95.25,25.5)(107.5,17)
\multiput(95.25,25.5)(.06920904,-.048022599){177}{\line(1,0){.06920904}}
%\end
\put(107.75,33.25){\line(1,0){14.75}}
%\emline(122.5,33.25)(133.25,25.25)
\multiput(122.5,33.25)(.064759036,-.048192771){166}{\line(1,0){.064759036}}
%\end
%\emline(133.25,25.5)(122.5,17)
\multiput(133.25,25.5)(-.060734463,-.048022599){177}{\line(-1,0){.060734463}}
%\end
%\emline(107.75,33)(122.75,17)
\multiput(107.75,33)(.0480769231,-.0512820513){312}{\line(0,-1){.0512820513}}
%\end
\put(75.31,17.37){\circle*{1.33}}
\put(42.81,17.37){\circle*{1.33}}
\put(90.31,17.37){\circle*{1.33}}
\put(57.81,17.37){\circle*{1.33}}
\put(90.31,30.37){\circle*{1.33}}
\put(57.81,30.37){\circle*{1.33}}
\put(90.31,30.37){\line(0,-1){13}}
\put(57.81,30.37){\line(0,-1){13}}
\put(90.31,17.37){\line(-1,0){15}}
\put(57.81,17.37){\line(-1,0){15}}
\qbezier(75.31,17.37)(75.81,18)(75.31,17.12)
\qbezier(42.81,17.37)(43.31,18)(42.81,17.12)
\qbezier(75.31,17.12)(75.19,18.12)(75.56,17.12)
\qbezier(42.81,17.12)(42.69,18.12)(43.06,17.12)
\qbezier(75.56,17.12)(74.69,18.37)(75.31,17.62)
\qbezier(43.06,17.12)(42.19,18.37)(42.81,17.62)
\put(74.81,29.87){\circle*{1.33}}
\put(42.31,29.87){\circle*{1.33}}
\put(74.81,30.37){\line(1,0){16}}
\put(42.31,30.37){\line(1,0){16}}
%\emline(90.81,30.37)(90.06,30.62)
\multiput(90.81,30.37)(-.125,.041667){6}{\line(-1,0){.125}}
%\end
%\emline(58.31,30.37)(57.56,30.62)
\multiput(58.31,30.37)(-.125,.041667){6}{\line(-1,0){.125}}
%\end
\put(90.06,30.62){\line(1,0){.25}}
\put(57.56,30.62){\line(1,0){.25}}
%\emline(57.56,17.62)(42.06,30.37)
\multiput(57.56,17.62)(-.058490566,.0481132075){265}{\line(-1,0){.058490566}}
%\end
\put(74.81,30.37){\line(0,-1){12.75}}
\put(42.31,30.37){\line(0,-1){12.75}}
\qbezier(75.06,29.87)(82.56,42.87)(90.06,30.87)
%\emline(57.56,30.62)(42.56,17.62)
\multiput(57.56,30.62)(-.0555555556,-.0481481481){270}{\line(-1,0){.0555555556}}
%\end
\qbezier(75.06,17.62)(81.56,5.25)(90.06,17.37)
\put(83.75,6){\makebox(0,0)[cc]{$G_{7}$}}
\put(51.25,6){\makebox(0,0)[cc]{$G_{6}$}}
\bezier{579}(74.75,30.25)(64,23)(75.25,17.75)
\put(82.75,39){\makebox(0,0)[cc]{$x$}}
\put(71.25,23.5){\makebox(0,0)[cc]{$y$}}
\put(82,13){\makebox(0,0)[cc]{$z$}}
\put(49.5,14.25){\makebox(0,0)[cc]{$y$}}
\bezier{501}(57.5,30.75)(60.5,42.875)(67.5,36.5)
\bezier{493}(67.5,36.5)(70.875,28.875)(57.75,30.75)
\put(62,41.25){\makebox(0,0)[cc]{$z$}}
\put(50,32.75){\makebox(0,0)[cc]{$x$}}
\put(21.915,17.745){\circle*{1.33}}
\put(36.915,17.745){\circle*{1.33}}
\put(36.915,30.745){\circle*{1.33}}
\put(36.915,30.745){\line(0,-1){13}}
\put(36.915,17.745){\line(-1,0){15}}
\qbezier(21.915,17.745)(22.415,18.375)(21.915,17.495)
\qbezier(21.915,17.495)(21.795,18.495)(22.165,17.495)
\qbezier(22.165,17.495)(21.295,18.745)(21.915,17.995)
\put(21.415,30.245){\circle*{1.33}}
\put(21.415,30.745){\line(1,0){16}}
%\emline(37.415,30.745)(36.665,30.995)
\multiput(37.415,30.745)(-.125,.041667){6}{\line(-1,0){.125}}
%\end
\put(36.665,30.995){\line(1,0){.25}}
%\emline(36.665,17.995)(21.165,30.745)
\multiput(36.665,17.995)(-.058490566,.0481132075){265}{\line(-1,0){.058490566}}
%\end
\put(21.415,30.745){\line(0,-1){12.75}}
%\emline(36.665,30.995)(21.665,17.995)
\multiput(36.665,30.995)(-.0555555556,-.0481481481){270}{\line(-1,0){.0555555556}}
%\end
\put(30.355,5.5){\makebox(0,0)[cc]{$G_{5}$}}
\qbezier(21.605,17.625)(29.98,6.5)(36.855,17.875)
\qbezier(21.355,30.125)(28.73,42.75)(36.605,30.875)
\put(139.765,30.87){\circle*{1.33}}
\put(-.235,30.62){\circle*{1.33}}
\put(156.105,30.87){\circle*{1.33}}
\put(16.105,30.62){\circle*{1.33}}
\put(156.105,17.87){\circle*{1.33}}
\put(16.105,17.62){\circle*{1.33}}
\put(139.765,17.87){\circle*{1.33}}
\put(-.235,17.62){\circle*{1.33}}
\put(147.765,41.87){\circle*{1.33}}
\put(7.765,41.62){\circle*{1.33}}
\put(139.765,30.87){\line(1,0){16.33}}
\put(-.235,30.62){\line(1,0){16.33}}
\put(156.105,30.87){\line(0,-1){13}}
\put(16.105,30.62){\line(0,-1){13}}
\put(156.105,17.87){\line(-1,0){16.33}}
\put(16.105,17.62){\line(-1,0){16.33}}
\put(139.765,17.87){\line(0,1){13}}
\put(-.235,17.62){\line(0,1){13}}
\put(139.765,30.87){\line(3,4){8.33}}
\put(-.235,30.62){\line(3,4){8.33}}
\put(147.765,41.87){\line(3,-4){8.33}}
\put(7.765,41.62){\line(3,-4){8.33}}
\put(147.765,41.87){\line(-1,-3){8}}
\put(7.765,41.62){\line(-1,-3){8}}
\put(147.765,41.87){\line(1,-3){8}}
\put(7.765,41.62){\line(1,-3){8}}
\put(147.375,6.25){\makebox(0,0)[cc]{$G_{9}$}}
\put(7.375,5.75){\makebox(0,0)[cc]{$G_{4}$}}
\qbezier(147.875,41.75)(160.125,40.25)(156.375,30.75)
\qbezier(147.875,41.75)(136.5,40.625)(139.625,31)
\put(122.5,33){\line(0,-1){16}}
%\emline(107.75,16.75)(133,25.25)
\multiput(107.75,16.75)(.142655367,.048022599){177}{\line(1,0){.142655367}}
%\end
%\emline(95.5,26)(122.5,17.25)
\multiput(95.5,26)(.148351648,-.048076923){182}{\line(1,0){.148351648}}
%\end
\put(109.5,28){\makebox(0,0)[cc]{$z$}}
\put(98.25,20.5){\makebox(0,0)[cc]{$y$}}
\put(129.5,19.25){\makebox(0,0)[cc]{$x$}}
%\emline(-.25,17.75)(16,30.5)
\multiput(-.25,17.75)(.0613207547,.0481132075){265}{\line(1,0){.0613207547}}
%\end
%\emline(-.5,30.75)(15.75,17.5)
\multiput(-.5,30.75)(.0590909091,-.0481818182){275}{\line(1,0){.0590909091}}
%\end
\put(2,36.75){\makebox(0,0)[cc]{}}
\put(7.75,14.5){\makebox(0,0)[cc]{}}
\put(29,39.25){\makebox(0,0)[cc]{}}
\put(29.5,14){\makebox(0,0)[cc]{}}
\put(137.5,39.75){\makebox(0,0)[cc]{$x$}}
\put(158.25,39.75){\makebox(0,0)[cc]{$y$}}
\put(147.75,15){\makebox(0,0)[cc]{$z$}}
\end{picture}

\end{center}

\vskip.4cm

The proof of the above theorem is  long and it is divided into four sections. In Section 2,  we prove the \textquotedblleft if part\textquotedblright of the theorem. The proof of remaining part is given in the next three sections.  In the last section, we discuss the extension of Theorem \ref{c2grspl3} for general $T$. 

\section{Graphic Splitting with respect to Three Elements}

The following result relates a splitting matroid and an elementary lift. 
\begin{lemma}\label{liftth}
	Let $M$ and $N$ be  binary matroids.  Then $N$ is an elementary lift of $M$ if and only if $N$ is isomorphic to $M_T$ for some $T \subseteq E(M)$.  
\end{lemma}
\begin{proof} 
Let $A$ be the standard matrix representation of a binary matroid $M$ and let $T\subseteq E(M)$.  
Let $B$ be the matrix obtained by adding a row to $A$ with entries zero everywhere except in the columns labeled by $T$ where it takes value $1$.  Then $B$ represents the matroid $M_T.$ Let $C$ be the matrix obtained by adding an extra column labeled by $a$ to $B$ containing 1 in the last row and zero otherwise. Suppose $Q$ is the vector matroid of $C$. Then $M=Q/a$ and $M_T= Q\backslash a$. Therefore the splitting matroid $M_T$ is an elementary lift of $M$. 

Conversely, suppose $N$ is a binary elementary lift of $M$. Then there exists
 a matroid $Q'$ such that $N=Q'\backslash a$ and $M=Q'/a$ for some $a \in E(Q')$. Let $T_1$ be a cocircuit of $Q'$ containing $a$.  Let $T=T_1-\{a\}$.   We can write the matrix $A$ of $Q'$  such that the column of $A$ labeled by $a$ has entry 1 in the last row and  zero elsewhere.  Since $T_1$ is  a cocircuit of $Q'$, the entries in the last row below the columns corresponding to $T_1$ are 1 and the remaining entries are zero.  Let $B$ be the matrix obtained from $A$ by deleting the last row and the column corresponding to $a$. Then $M[B]=Q'/a=M$.  Let $C$ be a matrix obtained from $A$ by deleting column corresponding the element $a$. Therefore $M[C]=Q'\backslash a=N$.   Thus, $C$ can be obtained from $B$ by adding one extra row which has entries 1 below the elements corresponding to $T$ and zero elsewhere.  Thus, by Definition \ref{c11},  $C=B_T$. Hence $N=M[C]=M[B_T]=M_T$.  
  \end{proof}

The following observations follow trivially from the definition of the splitting operation. 
\begin{lemma} \label{c2dcm}
	Let $M$ be a binary matroid and  $T \subseteq E(M).$   Then the following holds.
	\begin{enumerate}[label=(\roman*).]
				\item If   $x \in E(M)- T,$ then  $(M_T)\backslash x =(M\backslash x)_T$ and $(M_T)/ x=(M/ x)_T.$
		\item If $ y \in T$,  then $ M_T\backslash y = (M \backslash y)_{T -\{y\}}.$
		\item $M_T\backslash T=M\backslash T.$
		\item $M_T=M$, if $T$ is a cocircuit of $M$.	
	\end{enumerate}
\end{lemma}

We use the following important result frequently.  
\begin{theorem}[\cite{oxley2006matroid}, p. 385] \label{c1cgm}
	A binary matroid is graphic if and only if it has no minor isomorphic to $F_7$, $F_7^*$, $M^*(K_{3,3})$ or $M^*(K_5)$. 
	\end{theorem} 

``If part" of Theorem \ref{c2grspl3} follows from the following proposition.			
\begin{proposition}\label{c2exten2}
	Let $M$ be a binary matroid containing  a minor isomorphic to  one of the circuit matroids  $M(\tilde{G_4}),  M(\tilde{G_5}) $ and $M(G_k)$ for $k=6,7,8,9$, where  $G_4, G_5, \dots, G_{9}$ are the graphs shown in Figure 2. Then there is $T \subseteq E(M)$ with $|T|=3$ such that $M_T$ is not graphic. 
\end{proposition} 
\begin{proof} We prove that for each of  these six circuit matroids contains a set $T$ of three elements such that splitting with respect to $T$ results into a non-graphic matroid. 	Note that $G_4=G_3$ and $G_5=G_1$, where $G_1$ and $G_3$ are the graphs as shown in Figure 1. The matroid $M(\tilde{G_{i}})$ is obtained by adding an element $z_i$ to $M(G_i)$ for $i=4,5$. Hence $M(\tilde{G_{i}})\backslash z_i=M(G_i)$.   By Theorem \ref{c2bm}, there exists $T_i \subset E(M(G_i))$ with $|T_i|=2$ such that $M(G_i)_{T_i}$ is not graphic for $i=4,5$. Let $T=T_i \cup z_i$. By Lemma \ref{c2dcm}(ii),   $M(\tilde{G_{i}})_T \backslash  z_i \cong (M(\tilde{G_{i}}) \backslash z_i)_{T_i} \cong M(G_i)_{T_i}.$  Hence $M(\tilde{G_{i}})_T\backslash z_i$ is not graphic and so $M(\tilde{G_{i}})_T$ is not graphic.

	 Suppose $M=M(G_6)$.  Let  $ T = \{x,y, z\},$ where $x,y,z$ are the edges of the graph $G_6$  as shown in Figure 2.  Then the following are the matrix representations of $M$ and  $M_T$, respectively.  
	\begin{center}
		$A = $\bordermatrix{    ~&x&y&z&~&~&~&~\cr
			~&1&0&0&0&1&0&1\cr
			~&0&1&0&0&1&1&0\cr
			~&0&0&0&1&1&1&1},
		$A_T = $\bordermatrix{  ~&x&y&z&~&~&~&~\cr
			~&1&0&0&0&1&0&1\cr
			~&0&1&0&0&1&1&0\cr
			~&0&0&0&1&1&1&1\cr
			~&1&1&1&0&0&0&0}.
	\end{center} 
	Let $A'$ be the matrix obtained from the matrix $A_T$ by performing row operation $R_4 \rightarrow R_4+(R_1+R_2)$ and then interchanging the third and fourth row of the resulting matrix. 
		\begin{center} 
		$A' = $\bordermatrix{  ~&~&~&~&~&~&~&~\cr
		~&1&0&0&0&1&0&1\cr
		~&0&1&0&0&1&1&0\cr
		~&0&0&1&0&1&1&1\cr
		~&0&0&0&1&0&1&1}.
\end{center}
	Then the matrix $A'$ represents the matroid $F_7^*.$  Thus, we have  $ M(G_{6})_T=M[A_T] \cong M[A'] \cong F^*_7.$  Therefore by Theorem \ref{c1cgm}, $M_T$ is not graphic if $M=M(G_6)$

	Let $ T_j = \{x, y, z\} ,$ where $x$, $y$, $z$ are the edges of the graph $ G_{j}$ as shown in the Figure 2 for  $j=7,8, 9$. From the matrix representations, one can easily check that $ M(G_{7})_{T_7} \cong  F^*_7,$  
	$M(G_{8})_{T_8} \cong M^*(K_5)$ and $ M(G_{9})_{T_9}/\{z\} \cong M^*(K_{3,3}).$ Hence, by Theorem \ref{c1cgm}, $M(G_j)_{T_j}$ is not graphic for $j=7,8,9.$

	Suppose $M$ contains a minor $N$ which is isomorphic to  one of the six circuit matroids as listed in the statement. Then $ N_T$ is not graphic for some $ T \subseteq E(N)$ with $|T| = 3.$ Let $ A , B \subseteq E(M)$ such that $ N \cong M\backslash A /B.$ Then $ T$ is disjoint from $A$ and $B.$   By repeated applications of Lemma \ref{c2dcm}(i), we have  $  M_T\backslash A /B  =  (M\backslash A /B)_T \cong N_T.$ Therefore $M_T$ is not graphic.  
\end{proof}

We give the proof of the converse part of  Theorem \ref{c2grspl3} using   the next three sections.

\section{Minimal Matroids}
Let $ \mathcal{F} = \{F_7, F_7^*, M^*(K_{3,3}), M^*(K_5)\}.$  We need the following lemmas.
\begin{lemma}  \label{c2minle}
	Let $M$ be a graphic matroid  such that  $M_T$ contains a minor isomorphic to $F$ for some $T \subseteq E(M)$ with $|T| = 3$ and $ F \in \mathcal{F}$.  Then $M$ has a minor $N$ containing $T$  such that one of the following holds.
	\begin{enumerate}[label=(\roman*).]
		\item   $N_T\cong F.$
		\item   $N_T/ T'\cong F$ for some non-empty subset $T'$ of $T.$ 
		\item $N$ is isomorphic to one of the matroids $M(\tilde{G_{1}})$, $M(\tilde{G_{2}})$ and $M(\tilde{G_{3}}),$ where $ G_{i}$ is the graph as shown in Figure 1 for $i=1,2,3$.
		
	\end{enumerate} 
\end{lemma}
\begin{proof}	
	Since $M_T$ has a minor isomorphic to $F,$ there are $ T_1, T_2 \subseteq E(M)$ such that $M_T \backslash T_1/T_2 \cong F.$  Let $ T_i' = T \cap T_i$ and $ T_i'' = T_i - T_i'$ for $ i = 1, 2.$ Then $T_i'$ is a subset of $T$ while $T_i''$ is disjoint from $T.$  By Proposition \ref{c2dcm}(i),  $M_T \backslash T_1''/T_2'' = (M \backslash T_1''/T_2'')_T.$ Let $N = M \backslash T_1''/T_2''.$ Then $N$ is a minor of $M$  containing $T$ such that $ F \cong M_T \backslash T_1 /T_2 = N_T\backslash T_1'/T_2'.$

	Suppose $ T_1' = \emptyset.$  Then $N_T /T_2' \cong F.$ If $T_2' = \emptyset ,$ then (i) holds, otherwise  (ii) holds.
	
	Suppose $ T_1'\neq \emptyset.$  We prove that  (iii) holds.  As $ T_1' \subseteq T,$ we have $|T_1'| = 1, 2$ or 3.   Assume that $|T_1'|  = 3.$ Then $ T_1' = T$ and  $ T_2' = \emptyset.$ Hence $ F \cong N_T \backslash T_1' =  N_T \backslash T = N \backslash T$ by Lemma \ref{c2dcm}(iii).  This shows that $F$ is a minor of $N$ and so it is a minor of the graphic matroid $M,$ a contradiction.
	
	Assume that $| T_1'| = 2.$ Then $T - T_1'$ consists of only one element, say $z.$ Clearly, $z$ is a coloop of  $N_T \backslash T_1'.$ However, being a 3-connected matroid, $F$ does not contain a coloop. Therefore $ z \in T_2',$  in fact, $T_2' = \{z\}.$ Thus $ F \cong N_T \backslash T_1' / T_2' =  N_T \backslash T_1' \backslash T_2' = N_T \backslash T = N \backslash T.$ In this case also we get a contradiction. 
	
	Thus $|T_1'| = 1.$ Let $T_1' = \{z\}$,  $X = T - T_1'$ and $P = N \backslash z.$   Then $P$ is a minor of $N$  and $|X| = 2.$ Further,  $P$ is graphic as $N$ is graphic. By Proposition \ref{c2dcm}(ii), $ N_T \backslash T_1' = N_T \backslash z = (N \backslash z)_{T - \{z\}} = P_X.$  Thus we have $ P_X/T_2' \cong F$ with $T_2'$ as a subset of $X.$   Hence this case  reduces to the case of graphic splitting by two elements. In \cite{shikare2010excluded}, it is proved  that such $P$ is isomorphic to one of the three matroids as stated in  Theorem \ref{c2bm}. Therefore $N$ is isomorphic to one of the three matroids as stated in (iii). \end{proof}

\begin{lemma} \label{c2nocckts}
	The minor $N$ of the matroid $M$ as stated in Lemma \ref{c2minle}   can be chosen so that it  does not contain a 2-cocircuit. 
\end{lemma}
\begin{proof}  The matroids $M(\tilde{G_{1}})$, $M(\tilde{G_{2}}),$  $M(\tilde{G_{3}})$ do not contain a 2-cocircuit. Also, $F$ being a 3-connected matroid does not contain a 2-cocircuit.   Assume that  $N$ contains a 2-cocircuit, say $X = \{x, y\}.$ Then, by Lemma \ref{c2minle}, there is a subset $T'$ of $T$ such that $ N_T/T'$ is isomorphic to $F.$   By definition of the splitting operation,  $X$ is a 2-cocircuit of $N_T$ also. Hence $ T'$ intersects $X$ as $F$ does not contain a 2-cocircuit.

	Suppose $X \subseteq T.$ Then $ T = \{x, y, z\}$ for some $ z \in E(N).$ Clearly, $z$ is a coloop in $N_T.$ Hence  $ z \in T'.$ By Proposition \ref{c2dcm}(ii), $ N_T/z = N_T \backslash z = (N\backslash z )_{X}.$  Since $X$ is a cocircuit in $N,$  either it remains a  cocircuit in $N\backslash z $  or both $ x$ and $y$ become coloops in  $N\backslash z .$  Suppose $ X$ is a cocircuit in $N\backslash z .$  Then  $(N\backslash z )_X  = N\backslash z.$ Therefore $ F \cong N_T/T'  = N_T/z /(T'-\{z\}) = N\backslash z/(T'-\{z\}).$ This shows that $F$ is a minor of the graphic matroid $N,$ a contradiction. Thus, both $ x$ and $y$ are coloops in $N\backslash z .$ Therefore, they are coloops in $(N\backslash z )_X = N_T/z.$  Since $ F$ does not contain a coloop, both $x$ and $y$ belong to $T'.$ Thus $ T' = T =\{x, y, z\}$  and $ F \cong N_T/T'  = N_T/z /X  = N\backslash z/X =  N\backslash z \backslash X = N \backslash T.$  Therefore the graphic matroid $N$ has $F$ as a minor, a contradiction.
	
	Suppose $X$ is not a subset of $T.$ Then $T$ contains one member of $X.$ We may assume that $ T$ contains $ x.$ Hence $ T'\cap X = T\cap X = \{x\}.$ Since $\{x, y\}$ is a cocircuit of $N_T,$ we have $ N_T/x \cong N_T/y = (N/y)_T.$  Therefore $F\cong N_T/T' = N_T/x/(T' -\{x\})  \cong (N/y)_T/(T' -\{x\}).$ Thus, in this case,  we may replace the minor $N$ of $M$ by $ N/y.$  If $ N/y$ contains a 2-cocircuit say $\{x_1, y_1\},$ then $ x \notin \{x_1, y_1\}$ and  as above $T$ contains only one of them, say $x_1.$ Then  $( N/y/y_1)_T/(T'-\{x, x_1\})\cong F.$ So replace $N/y$ by $N/\{y, y_1\}.$ Similarly, if $N/\{y, y_1\}$ contains a 2-cocircuit, say $\{x_2, y_2\},$ then $T$ contains only one of them, say $x_2$ and  we have $ F \cong (N/\{y, y_1, y_2 \})_T.$  Note that  $N/\{y, y_1, y_2 \}$ does not contain a 2-cocircuit and so we can replace $N$ by $N/\{y, y_1, y_2 \}.$
	This completes the proof. \end{proof}

\begin{lemma}\label{c2nocirut}
	If the minor $N$ of the matroid $M$ as stated in Lemma \ref{c2minle}   does not satisfy Condition (iii), then it does not contain a coloop. 
\end{lemma}
\begin{proof}
	Suppose the minor $N$ of $M$ does not satisfy Condition (iii). Then $N$ satisfies Condition (i) or (ii).  Therefore $N_T \cong F$ or $ N_T/T'\cong F$ for some $ T' \subset T.$ Suppose $N$ contains a coloop, say $z.$ Then $ z$ is a coloop in $N_T.$  Since $ F$ does not contain a coloop,  $N_T \ncong F$.  Also,  $N_T/z = N_T\backslash z = (N\backslash z)_{T-\{z\}}.$  Therefore $ z\in T'.$  Hence $ F \cong N_T/T' = (N\backslash z)_{T-\{z\}} /(T' -\{z\}).$ This case reduces to the case of splitting by two elements.  Hence, by Theorem \ref{c2bm}, $N$ is isomorphic to one of the three matroids $M(\tilde{G_{1}})$, $M(\tilde{G_{2}})$ and $M(\tilde{G_{3}}).$ Thus $N$ satisfies Condition (iii) of Lemma \ref{c2minle}, a contradiction. 
\end{proof}

We are trying to prove that the six matroids as stated in Theorem \ref{c2grspl3} are the forbidden minors for the class of graphic matroids whose splitting with respect to three elements is graphic. Out of these six, we get two from the case of splitting by two elements as seen in Condition (iii) of Lemma \ref{c2minle}.   To get the remaining forbidden minors, we need only to consider the graphic matroid $N$ satisfying Condition (i) or (ii) of Lemma \ref{c2minle}. Then $N$ also satisfy Lemmas \ref{c2nocckts} and \ref{c2nocirut}. We call such matroids as minimal matroids with respect to $F \in \mathcal{F}.$  We define them formally as follows.
\begin{definition} \label{c2mindef}
	Let $N$ be a graphic  matroid and let $ F \in \mathcal{F}.$ We say that $ N$ is \textit{minimal} with respect to $F$ if the following conditions hold.
	\begin{enumerate}[label=(\roman*).]
	\item $N$ contains no coloop and no 2-cocircuit.
	\item  There is $ T \subseteq E(N)$ with $|T|=3$ such that $ N_T \cong F$ or $ N_T / T' \cong F$ for some non-empty subset $T'$ of $T.$   
	\end{enumerate}
	\end{definition}

We need to find the minimal matroids with respect to $F$ for every $ F \in \mathcal{F}.$

We relate a  minimal matroid $M$ with respect to $F$ with a particular type of minor $P$ so that $ M = P$ or $M$ is a coextension of $P$ by one, two or three elements.  To get such a minor $P,$ we first provide below the definition of  a binary matroid $M_T'$ which is a coextension of $M$ by an element $a$ and is an extension  of the splitting matroid  $M_T$  by $a.$   
\begin{definition} \cite{azanchiler} \label{c2esd} Let $M$ be a binary matroid with standard matrix representation $A$ over the field $GF(2)$ and let $T \subseteq  E(M).$ Let $A'_T$ be the matrix obtained from $A$ by adjoining one extra row to the matrix $A$ whose entries are 1 in the columns labeled by the elements of $T$ and zero otherwise and then adding one extra column labeled by $a$ with entry 1 in the last row and 0 elsewhere. Denote the vector matroid of $A'_T$ by $M'_T$.  \end{definition}

The ground set of $M_T'$ is $E(M)\cup\{a\}.$ From the definition of $M_T$ and $M_T',$ the following relations   follow immediately.  
\begin{lemma} \cite{azanchiler} \label{c2esc}
	Let $M$ be a binary matrod and $T \subseteq E(M).$    Then  
	\begin{enumerate}[label=(\roman*).]
		\item $M_T  =  M'_T \backslash a$ and 
		\item $M'_T /a \cong M.$ 
	\end{enumerate}
\end{lemma}

Using Definition \ref{c2esd}, we prove in the following lemma  that a minimal matroid with respect to  $F $ has a minor $P$ with the property that  there is a binary  coextension $N$ of $P$ by an element $a$ such that  $ N\backslash a \cong F.$ 
\begin{lemma} \label{c2mml}
	Let $M$ be a graphic matroid which is  minimal with respect to $F \in \mathcal{F}.$ Then there is a binary matroid $N$ containing an element    $a$ such that $N \backslash a \cong F.$ Further,  $N/a$ is a minor of $M$ such that either $M = N/a $ or $M$ is a coextension of $N/a$ by one, two or three elements.
\end{lemma}

\begin{proof}  There is $T \subseteq E(M)$ with $|T| = 3$ such that $M_T \cong F$ or $M_T/T' \cong F$ for some nonempty subset $T'$ of $T.$ 
	
	\vskip.3cm
	(i).\hskip.2cm Suppose $M_T \cong F.$ Let $N=M'_T.$ By Lemma \ref{c2esc}(i), $N$ is a binary matroid containing an element $a$ such that    $N\backslash a  = M_T \cong F$ and $ N /a = M.$

	\vskip.5cm
	\begin{center}
	% This is a LaTeX picture output by TeXCAD.
	% File name: [fig element splitting.pic].
	% Version of TeXCAD: 4.3
	% Reference / build: 30-Jun-2012 (rev. 105)
	% For new versions, check: http://texcad.sf.net/
	% Options on the following lines.
	%\grade{\on}
	%\emlines{\off}
	%\epic{\off}
	%\beziermacro{\on}
	%\reduce{\on}
	%\snapping{\off}
	%\pvinsert{% Your \input, \def, etc. here}
	%\quality{8.000}
	%\graddiff{0.005}
	%\snapasp{1}
	%\zoom{4.0000}
	\unitlength 1mm % = 2.845pt
	\linethickness{0.4pt}
	\ifx\plotpoint\undefined\newsavebox{\plotpoint}\fi % GNUPLOT compatibility
	\begin{picture}(30.25,40)(0,0)
	\put(25.5,9.75){\makebox(0,0)[cc]{$M$}}
	\put(25,40){\makebox(0,0)[cc]{$N=M_T'$}}
	\put(3.5,24.25){\makebox(0,0)[cc]{$M_T$}}
	\put(11,35.25){\makebox(0,0)[cc]{$N\backslash a$}}
	\put(30.25,25.25){\makebox(0,0)[cc]{$N/a$}}
	\thicklines
	\put(25,36.25){\vector(0,-1){22}}
	%\vector(21.25,38)(8,28)
	\put(8,28){\vector(-4,-3){.07}}\multiput(21.25,38)(-.0446127946,-.0336700337){297}{\line(-1,0){.0446127946}}
	%\end
	%\vector(22,14.25)(7.75,24.75)
	\put(7.75,24.75){\vector(-4,3){.07}}\multiput(22,14.25)(-.0456730769,.0336538462){312}{\line(-1,0){.0456730769}}
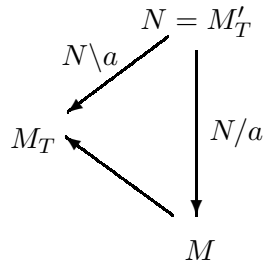
		\put(13,2){\makebox(0,0)[cc]{Figure: Relation between $N$, $M_T$ and $M$}}
	%\end
	\end{picture}
		\end{center}

	(ii).\hskip.2cm Suppose $M_T/T' \cong F$ for some nonempty subset $T'$ of $T.$ 
	In this case, define $ N = M_T'/T'.$  Then $N$ is a binary matroid containing $a$ such that    $N\backslash a  = M_T'\backslash a /T'= M_T/T'\cong F .$ Further,  $ N /a = M_T'/a/T' = M/T'.$ Hence $M$ is a coextension of $N/a$ by $T'.$ As $|T'| \leq 3,$ the result follows. 
\end{proof}

\section{Graphic Quotients of Non-graphic Binary Matroids}
By Lemma \ref{c2mml}, if a graphic matroid $M$ is minimal with respect to $F$ for some $F \in \mathcal{F}=\{F_7, F_7^*, M^*(K_{3,3}), M^*(K_5)\},$ then there is a binary matroid $N$ such that $N/a$ is a minor of $M$ and $N\backslash a \cong F$.  Hence $N/a$ is a graphic quotient of $F$. In this section, we find $N/a$ for every matroid of the collection $\mathcal{F}$.

We first prove the  following lemma.
\begin{lemma}\label{c2slemma}
	Let $N$ be a binary matroid  and $a$ be an element of $N$ such that  $N\backslash a \cong F$ for $ F \in \mathcal{F}$ and $N/ a$ is the circuit matroid of a connected graph  $G.$ Then the following holds.
	
	\begin{enumerate}[label=(\roman*).]		
		\item  $G$ is a block or it has two blocks one of which is a loop.
		\item $G$ contains at most one loop.
		\item $G$ does not contain more than two edges with same pair of end vertices.
		\item $G$ does not contain a 2-edge cut. 
		\item If $F$ is an Eulerian matroid, then $G$ is an Eulerian graph.
	\end{enumerate}
\end{lemma}

\begin{proof}  
	 Since $F$ is 3-connected,  $F$  does not contain a loop, a coloop, a 2-circuit and a 2-cocircuit.    Suppose $a$ is a loop or coloop of $N.$ Then $ F \cong N\backslash a = N/a$ and so, by Theorem \ref{c1cgm}, $N/a$ is not graphic, a contradiction. Hence $a$ belongs to a circuit of $N$ of size at least two. Since $ N \backslash a$ is 3-connected,  $N$ is connected.
	
	\vskip.5cm
	$(i).$\hskip.3cm Suppose $G$ is not a block.  Let $B_1$ be a block of $G$ containing the maximum number of edges of $G$.   Suppose $B_1$  consists of only one edge. Then every block of $G$ consists of only one edge. Hence every element of $ N/a$ is a loop or a coloop. This implies that all elements of $N$ are parallel to each other. Hence all elements of $N\backslash a$ are parallel to each other. As $F$ has at least seven elements, $N\backslash a$ also has at least seven elements. This shows that $F$ contains a 2-circuit, a contradiction. 
	
	Hence $B_1$ consists of at least two edges. Let $B_2$ be the  union  of all blocks of $G$ other than $B_1.$  Suppose $B_2$ also has at least two edges. 	 Let $r$ be the rank function of $N.$ Let $B'_1$ and $B'_2$ be the  minors of $N$ each containing the element $a$ such that $ B_i'/a = B_i$ for $ i = 1, 2.$  Then $r(B_i'\backslash a) \leq r(B_i) + 1$ for $ i = 1, 2.$ Also, $ r(N\backslash a)  = r(N)  = r(N/a) + 1.$  Since $G$ is not $2$-connected, $N/a$ is disconnected and $r(B_1)+r(B_2)-r(N /a)=0$. 
	Hence 
	$$r(B'_1 \backslash a)+r(B'_2 \backslash a)- r(N\backslash a)\leq  r(B_1)+1+r(B_2)+1-r(N/ a)-1=1$$

	Thus $(E(B_1), E(B_2))$  gives a 2-separation of $N\backslash a.$ Therefore $N\backslash a$ is not 3-connected, a contradiction.
	Thus $|E(B_2)| = 1$ and  hence $B_2$ is a block of $G$  and it must be a loop of $G.$  
	
	\vskip.5cm
	 $(ii).$\hskip.3cm Suppose $G$ contains two loops, say $ x$ and $y.$ Then $x$ or $y$ or both are loops in $N$ or  $ \{a, x\}$ and $ \{a, y\}$ and so $\{x, y\}$ are circuits in $N.$ 
	Hence $ N\backslash a$ contains a loop or 2-circuit.  This shows that $F$ contains a loop or a 2-circuit, a contradiction. Thus $G$ contains at most one loop.
	
	\vskip.5cm
	$(iii).$\hskip.3cm  Assume that $G$ has three edges which are parallel to each other. It follows that at least two of them are in a $2$-circuit of $N.$ Hence $N\backslash a$ contains a 2-circuit, a contradiction. 
	
	\vskip.5cm
	 $(iv).$\hskip.3cm  Assume that $G$ contains a 2-edge cut  $\{x, y\}.$ Then $\{x, y\}$  is a 2-cocircuit in $N.$  Therefore $\{x, y\}$ contains a cocircuit of $N\backslash a.  $ Hence $F$ has a 1-cocircuit or a 2-cocircuit, a contradiction.  
	
	\vskip.5cm
	$(v).$\hskip.3cm Suppose $F$ is an Eulerian matroid. Then $N \backslash a$ is also Eulerian and so its dual $(N \backslash a)^* = N^* /a$ is a bipartite matroid.  Therefore   $N^* /a$ does not contain a circuit of odd size. Hence every circuit of $N^*$ which avoids  $a$ has even size. This shows that every circuit of $N^* \backslash a$ has even size.  Therefore $N^* \backslash a$ is a bipartite. Hence its dual $N/a$ is an Eulerian matroid. Consequently, the corresponding graph $G$ is an Eulerian graph. 
\end{proof}

We now find graphic quotients of $F$ for every $F \in \{F_7, F^*_7, M(K^*_{3,3}), M^*(K_5)\}.$

The following result give graphic elementary quotients of the matroid $F_7^*$.

\begin{lemma} \label{c2mff*7}
	Let $N$ be a binary matroid and $a \in E(N)$. If $N \backslash a \cong  F^*_7$ and $N /a$ is a graphic matroid, then $N/a$ is isomorphic to  $M(G_{10})$ or $M(G_{11})$, where $G_{10}$ and $G_{11}$ are the graphs as shown in Figure 3. 
\end{lemma}
\begin{center}
	% This is a LaTeX picture output by TeXCAD.
	% File name: [graphic quotient.pic].
	% Version of TeXCAD: 4.3
	% Reference / build: 30-Jun-2012 (rev. 105)
	% For new versions, check: http://texcad.sf.net/
	% Options on the following lines.
	%\grade{\on}
	%\emlines{\off}
	%\epic{\off}
	%\beziermacro{\on}
	%\reduce{\on}
	%\snapping{\off}
	%\pvinsert{% Your \input, \def, etc. here}
	%\quality{8.000}
	%\graddiff{0.005}
	%\snapasp{1}
	%\zoom{4.0000}
	\unitlength .7mm % = 1.992pt
	\linethickness{0.4pt}
	\ifx\plotpoint\undefined\newsavebox{\plotpoint}\fi % GNUPLOT compatibility
	\begin{picture}(157.83,42.875)(0,0)
	\put(70,-4){\makebox(0,0)[cc]{Figure 3}}
	\put(37.06,17.37){\circle*{1.33}}
	\put(4.56,17.37){\circle*{1.33}}
	\put(52.06,17.37){\circle*{1.33}}
	\put(19.56,17.37){\circle*{1.33}}
	\put(52.06,30.37){\circle*{1.33}}
	\put(19.56,30.37){\circle*{1.33}}
	\put(52.06,30.37){\line(0,-1){13}}
	\put(19.56,30.37){\line(0,-1){13}}
	\put(52.06,17.37){\line(-1,0){15}}
	\put(19.56,17.37){\line(-1,0){15}}
	\qbezier(37.06,17.37)(37.56,18)(37.06,17.12)
	\qbezier(4.56,17.37)(5.06,18)(4.56,17.12)
	\qbezier(37.06,17.12)(36.94,18.12)(37.31,17.12)
	\qbezier(4.56,17.12)(4.44,18.12)(4.81,17.12)
	\qbezier(37.31,17.12)(36.44,18.37)(37.06,17.62)
	\qbezier(4.81,17.12)(3.94,18.37)(4.56,17.62)
	\put(36.56,29.87){\circle*{1.33}}
	\put(4.06,29.87){\circle*{1.33}}
	\put(36.56,30.37){\line(1,0){16}}
	\put(4.06,30.37){\line(1,0){16}}
	%\emline(52.56,30.37)(51.81,30.62)
	\multiput(52.56,30.37)(-.125,.041667){6}{\line(-1,0){.125}}
	%\end
	%\emline(20.06,30.37)(19.31,30.62)
	\multiput(20.06,30.37)(-.125,.041667){6}{\line(-1,0){.125}}
	%\end
	\put(51.81,30.62){\line(1,0){.25}}
	\put(19.31,30.62){\line(1,0){.25}}
	%\emline(19.31,17.62)(3.81,30.37)
	\multiput(19.31,17.62)(-.058490566,.0481132075){265}{\line(-1,0){.058490566}}
	%\end
	\put(36.56,30.37){\line(0,-1){12.75}}
	\put(4.06,30.37){\line(0,-1){12.75}}
	\qbezier(36.81,29.87)(44.31,42.87)(51.81,30.87)
	%\emline(19.31,30.62)(4.31,17.62)
	\multiput(19.31,30.62)(-.0555555556,-.0481481481){270}{\line(-1,0){.0555555556}}
	%\end
	\qbezier(36.81,17.62)(43.31,5.25)(51.81,17.37)
	\put(45.5,6){\makebox(0,0)[cc]{$G_{11}$}}
	\put(13,6){\makebox(0,0)[cc]{$G_{10}$}}
	\bezier{579}(36.5,30.25)(25.75,23)(37,17.75)
	\put(44.5,39){\makebox(0,0)[cc]{}}
	\put(33,23.5){\makebox(0,0)[cc]{}}
	\put(43.75,14){\makebox(0,0)[cc]{}}
	\put(11.25,14.25){\makebox(0,0)[cc]{}}
	\bezier{501}(19.25,30.75)(22.25,42.875)(29.25,36.5)
	\bezier{493}(29.25,36.5)(32.625,28.875)(19.5,30.75)
	\put(24.75,27.5){\makebox(0,0)[cc]{}}
	\put(11.75,32.75){\makebox(0,0)[cc]{}}
	\put(70.935,32.495){\circle*{1.33}}
	\put(61.185,17.995){\circle*{1.33}}
	\put(79.685,17.495){\circle*{1.33}}
	\put(70.375,6){\makebox(0,0)[cc]{$G_{12}$}}
	%\put(31.5,15){\makebox(0,0)[cc]{$x_1$}}
	%\emline(61.125,17.875)(70.875,32.625)
	\multiput(61.125,17.875)(.048029557,.072660099){203}{\line(0,1){.072660099}}
	%\end
	%\emline(70.875,32.625)(79.375,17.375)
	\multiput(70.875,32.625)(.048022599,-.086158192){177}{\line(0,-1){.086158192}}
	%\end
	%\emline(79.375,17.375)(60.875,17.625)
	\multiput(79.375,17.375)(-3.083333,.041667){6}{\line(-1,0){3.083333}}
	%\end
	\qbezier(70.875,32.625)(84,31.125)(79.625,17.625)
	\qbezier(70.875,32.625)(58.75,31.375)(61.125,18.125)
	\qbezier(61.125,18.125)(69,7)(79.375,17.375)
	\qbezier(79.375,17.875)(80.375,5.75)(88.375,11.125)
	\qbezier(88.375,10.625)(88.875,18.625)(79.375,17.625)
	\put(99.06,16.62){\circle*{1.33}}
	\put(114.06,16.62){\circle*{1.33}}
	\put(114.06,29.62){\circle*{1.33}}
	\put(114.06,29.62){\line(0,-1){13}}
	\put(114.06,16.62){\line(-1,0){15}}
	\qbezier(99.06,16.62)(99.56,17.25)(99.06,16.37)
	\qbezier(99.06,16.37)(98.94,17.37)(99.31,16.37)
	\qbezier(99.31,16.37)(98.44,17.62)(99.06,16.87)
	\put(98.56,29.12){\circle*{1.33}}
	\put(98.56,29.62){\line(1,0){16}}
	%\emline(114.56,29.62)(113.81,29.87)
	\multiput(114.56,29.62)(-.125,.041667){6}{\line(-1,0){.125}}
	%\end
	\put(113.81,29.87){\line(1,0){.25}}
	\put(98.56,29.62){\line(0,-1){12.75}}
	\qbezier(98.81,29.12)(106.31,42.12)(113.81,30.12)
	\put(107.25,6){\makebox(0,0)[cc]{$G_{13}$}}
	\bezier{579}(98.5,29.5)(87.75,22.25)(99,17)
	\put(98.75,16.5){\line(-1,0){.75}}
	%\emline(98,16.5)(114,29.5)
	\multiput(98,16.5)(.0592592593,.0481481481){270}{\line(1,0){.0592592593}}
	%\end
	\qbezier(98.5,28.75)(99.75,17.625)(114,17)
	\qbezier(98.5,29.25)(110.375,28.375)(113.75,17)
	%\put(31.5,15){\makebox(0,0)[cc]{$x_1$}}
	\put(131.665,16.37){\circle*{1.33}}
	\put(146.665,16.37){\circle*{1.33}}
	\put(146.665,32.37){\circle*{1.33}}
	\put(131.665,32.37){\circle*{1.33}}
	\put(157.165,24.62){\circle*{1.33}}
	\put(119.415,24.87){\circle*{1.33}}
	\put(146.665,16.37){\line(-1,0){15}}
	\qbezier(131.665,16.37)(132.165,17)(131.665,16.12)
	\qbezier(131.665,16.12)(131.545,17.12)(131.915,16.12)
	\qbezier(131.915,16.12)(131.045,17.37)(131.665,16.62)
	\put(139.355,6){\makebox(0,0)[cc]{$G_{14}$}}
	%\emline(119.355,25)(131.605,32.25)
	\multiput(119.355,25)(.081125828,.048013245){151}{\line(1,0){.081125828}}
	%\end
	%\emline(119.105,24.75)(131.355,16.25)
	\multiput(119.105,24.75)(.06920904,-.048022599){177}{\line(1,0){.06920904}}
	%\end
	\put(131.605,32.5){\line(1,0){14.75}}
	%\emline(146.355,32.5)(157.105,24.5)
	\multiput(146.355,32.5)(.064759036,-.048192771){166}{\line(1,0){.064759036}}
	%\end
	%\emline(157.105,24.75)(146.355,16.25)
	\multiput(157.105,24.75)(-.060734463,-.048022599){177}{\line(-1,0){.060734463}}
	%\end
	%\emline(131.605,32.25)(146.605,16.25)
	\multiput(131.605,32.25)(.0480769231,-.0512820513){312}{\line(0,-1){.0512820513}}
	%\end
	\put(146.355,32.25){\line(0,-1){16}}
	%\emline(131.605,16)(156.855,24.5)
	\multiput(131.605,16)(.142655367,.048022599){177}{\line(1,0){.142655367}}
	%\end
	%\emline(119.355,25.25)(146.355,16.5)
	\multiput(119.355,25.25)(.148351648,-.048076923){182}{\line(1,0){.148351648}}
	%\end
	\put(133.355,27.25){\makebox(0,0)[cc]{}}
	\put(122.105,19.75){\makebox(0,0)[cc]{}}
	\put(153.355,18.5){\makebox(0,0)[cc]{}}
	\end{picture}
	
\end{center}
\vskip.2cm
\begin{proof}  Suppose  $N\backslash a\cong F^*_7$ and $N/a$ is isomorphic to  $M(G)$ for some connected graph $G$. 	If $a$ is a loop or coloop of $N$, then $N /a = N \backslash a \cong F^*_7$, a contradiction to the fact that $N/a$ is graphic.  Hence $a$ belongs to a circuit of $N$ of size greater than one. 	Since  $F^*_7$ has rank 4 and it contains 7 elements, the graph $G$ has $4$ vertices and $7$ edges.  By Lemma \ref{c2slemma}(i), $G$ is $2$-connected or has  two blocks  one of which is a loop. 
	
	\vskip.5cm
	\noindent {\bf Case (i).} Suppose  $G$ does not contain a loop. 
	
	Then $G$ is 2-connected. By Lemma \ref{c2slemma}(iii) and (iv), $G$ does not have a vertex of degree two and   more than two edges with same pair of end vertices.  Hence  $G$ can be obtained from   $H_i$ by adding one parallel edge, or from $H_{ii}$  by adding two parallel edges or from $H_{iii}$  by adding three parallel edges, where $H_i,$ $H_{ii}$ and $H_{iii}$ are the graphs as shown in Figure 4.

	Suppose $G$ is obtained from $H_{i}$ or $H_{ii}.$  Then $N$ contains a $3$-circuit or a $2$-circuit without containing $a$.   Therefore $N\backslash a$ contains a $3$-circuit or  a $2$-circuit, which a contradiction to the fact that  $F^*_7$ is a $3$-connected bipartite matroid.  If $G$ is obtained from $H_{iii},$ then  it is isomorphic to the graph $G_{11}.$  
	
	\vskip.5cm
	\noindent	{\bf Case (ii).} 	Suppose $G$ contains a loop.
	
	Then $G$ can be obtained from $H_i$ by adding a loop, or from $H_{ii}$ by adding a loop and a parallel edge or from $H_{iii}$ by adding  a loop and  two parallel edges.  If  $G$ is  obtained $H_{ii}$ or $H_{iii}, $ then $G$ contains a $2$-edge cut, a contradiction to  Lemma \ref{c2slemma}(iv). Hence  $G$ is obtained from $H_i. $ In this case, $G$  is isomorphic to the graph $G_{10}.$    
\end{proof} 

\begin{center}
	% This is a LaTeX picture output by TeXCAD.
	% File name: [c2 K4 minors.pic].
	% Version of TeXCAD: 4.3
	% Reference / build: 30-Jun-2012 (rev. 105)
	% For new versions, check: http://texcad.sf.net/
	% Options on the following lines.
	%\grade{\on}
	%\emlines{\off}
	%\epic{\off}
	%\beziermacro{\on}
	%\reduce{\on}
	%\snapping{\off}
	%\pvinsert{% Your \input, \def, etc. here}
	%\quality{8.000}
	%\graddiff{0.005}
	%\snapasp{1}
	%\zoom{4.0000}
	\unitlength .7mm % = 1.992pt
	\linethickness{0.4pt}
	\ifx\plotpoint\undefined\newsavebox{\plotpoint}\fi % GNUPLOT compatibility
	\begin{picture}(71.475,33.285)(0,0)
	\put(6.06,19.62){\circle*{1.33}}
	\put(31.31,19.62){\circle*{1.33}}
	\put(55.81,19.37){\circle*{1.33}}
	\put(21.06,19.62){\circle*{1.33}}
	\put(46.31,19.62){\circle*{1.33}}
	\put(70.81,19.37){\circle*{1.33}}
	\put(21.06,32.62){\circle*{1.33}}
	\put(46.31,32.62){\circle*{1.33}}
	\put(70.81,32.37){\circle*{1.33}}
	\put(21.06,32.62){\line(0,-1){13}}
	\put(46.31,32.62){\line(0,-1){13}}
	\put(70.81,32.37){\line(0,-1){13}}
	\put(21.06,19.62){\line(-1,0){15}}
	\put(46.31,19.62){\line(-1,0){15}}
	\put(70.81,19.37){\line(-1,0){15}}
	\qbezier(6.06,19.62)(6.56,20.25)(6.06,19.37)
	\qbezier(31.31,19.62)(31.81,20.25)(31.31,19.37)
	\qbezier(55.81,19.37)(56.31,20)(55.81,19.12)
	\qbezier(6.06,19.37)(5.94,20.37)(6.31,19.37)
	\qbezier(31.31,19.37)(31.19,20.37)(31.56,19.37)
	\qbezier(55.81,19.12)(55.69,20.12)(56.06,19.12)
	\qbezier(6.31,19.37)(5.44,20.62)(6.06,19.87)
	\qbezier(31.56,19.37)(30.69,20.62)(31.31,19.87)
	\qbezier(56.06,19.12)(55.19,20.37)(55.81,19.62)
	\put(5.56,32.12){\circle*{1.33}}
	\put(30.81,32.12){\circle*{1.33}}
	\put(55.31,31.87){\circle*{1.33}}
	\put(5.56,32.62){\line(1,0){16}}
	\put(30.81,32.62){\line(1,0){16}}
	\put(55.31,32.37){\line(1,0){16}}
	%\emline(21.56,32.62)(20.81,32.87)
	\multiput(21.56,32.62)(-.09375,.03125){8}{\line(-1,0){.09375}}
	%\end
	%\emline(46.81,32.62)(46.06,32.87)
	\multiput(46.81,32.62)(-.09375,.03125){8}{\line(-1,0){.09375}}
	%\end
	%\emline(71.31,32.37)(70.56,32.62)
	\multiput(71.31,32.37)(-.09375,.03125){8}{\line(-1,0){.09375}}
	%\end
	\put(20.81,32.87){\line(1,0){.25}}
	\put(46.06,32.87){\line(1,0){.25}}
	\put(70.56,32.62){\line(1,0){.25}}
	%\emline(20.81,19.87)(5.31,32.62)
	\multiput(20.81,19.87)(-.041005291,.0337301587){378}{\line(-1,0){.041005291}}
	%\end
	%\emline(46.06,19.87)(30.56,32.62)
	\multiput(46.06,19.87)(-.041005291,.0337301587){378}{\line(-1,0){.041005291}}
	%\end
	\put(5.56,32.62){\line(0,-1){12.75}}
	\put(30.81,32.62){\line(0,-1){12.75}}
	\put(55.31,32.37){\line(0,-1){12.75}}
	%\emline(20.81,32.87)(5.81,19.87)
	\multiput(20.81,32.87)(-.0388601036,-.0336787565){386}{\line(-1,0){.0388601036}}
	%\end
	\put(38.5,5){\makebox(0,0)[cc]{Figure 4}}
	\put(14.5,12.5){\makebox(0,0)[cc]{$H_i$}}
	\put(39.75,12.5){\makebox(0,0)[cc]{$H_{ii}$}}
	\put(64.25,12.25){\makebox(0,0)[cc]{$H_{iii}$}}
	\end{picture}
\end{center}

We now determine graphic elementary quotients of the matroid $F_7$.

\begin{lemma} \label{c2mff7}
	Let $N$ be a binary matroid and $a \in E(N)$. If $N\backslash a \cong F_7$ and $N/a$ is a graphic matroid, then $N/a$ is  isomorphic to the circuit matroid of $G_{12}$, where $G_{12}$ is the graph as shown in Figure 3.
\end{lemma}

\begin{proof} 
	
	Since  $N$ is a binary extension of $F_7$ by the element $a$, it follows from the matrix representation of $F_7$ that there is a 2-cocircuit $\{x,a\}$ in $N$ for some $x \in E(N)$.  Therefore  $x$ is a loop in $N /a.$   
	Hence $N /a \backslash x  =   \cong (N\backslash a) /x \cong  F_7 / x$.  However, the matroid  $F_7/x$ is isomorphic to the circuit matroid of a fat triangle,  a graph obtained from a triangle by adding three edges parallel to three edges of the triangle. Hence $G$ is obtained by adding a loop to a fat triangle and thus it is isomorphic to the graph $G_{12}.$ 
\end{proof}

By Proposition \ref{c2exten2}, $M(G_6)$ and $M(G_7)$ are the forbidden minors for the graphic splitting matroid $M_T$ of graphic matroids $M$ where $|T|=3$.
  
  %Also, by Lemmas \ref{c2mml} and \ref{c2mff*7}, whenever $M_T$ contains a minor $F^*_7$ then $M$ contains $M(G_6)$ or $M(G_7)$ as a minor. Therefore, we avoid the class of binary matroids which contain $M(G_6)$ or $M(G_7)$ as a minor whenever we try to obtain $M$ whose splitting matroid $M_T$ contains $M^*(K_5)$ or $M^*(K_{3,3})$ as a minor.  

In the next two lemmas, we find graphic quotients of $M^*(K_{3,3})$ and $M^*(K_{5})$ which avoids $M(G_6)$ and $M(G_7)$.

\begin{lemma} \label{c2mfm*33} 
	Let $N$ be a binary matroid with  $a \in E(N)$ and $N/ a$ does not contain a minor isomorphic to $M(G_6)$ or $M(G_7)$.  If  $N\backslash a \cong M^*(K_{3,3})$ and $N/a$ is a graphic matroid, then  $N/a$ is  isomorphic to the circuit matroid of $G_{13}$, where $G_{13}$ is the graph as shown in Figure 3.
\end{lemma} 

\begin{proof} Suppose $N/a=M(G)$ for some connected graph $G$. Since $M^*(K_{3,3,})$ has rank $4$ and 9 elements,  $N$ has rank $4$ and 10 elements.  Hence $G$ has $4$ vertices and $9$ edges.  Since $K_{3,3}$ is a bipartite graph, $M^*(K_{3,3})$ is an Eulerian matroid. Hence,  by Lemma \ref{c2slemma}(v), $G$ is an Eulerian graph. 		By Lemma \ref{c2slemma}(i), $G$ is $2$-connected or it is a block plus a loop.

	\vskip.5cm
	\noindent \textbf{{Case (i).}} Suppose $G$ contains a loop. 
	
	Then the block of  $G$ other than a loop must be an Eulerian graph on $4$ vertices and $8$ edges.   By Harary [\cite{harary6graph}, p.230],  every 2-connected Eulerian graph with 4 vertices and 8 edges is isomorphic to   $H_{iv}$ or $H_{v},$ where $H_{iv}$ or $H_{v}$ are the graphs as shown in Figure 5. 
	\begin{center} 
		% This is a LaTeX picture output by TeXCAD.
		% File name: [C2 ford dual m33.pic].
		% Version of TeXCAD: 4.3
		% Reference / build: 30-Jun-2012 (rev. 105)
		% For new versions, check: http://texcad.sf.net/
		% Options on the following lines.
		%\grade{\on}
		%\emlines{\off}
		%\epic{\off}
		%\beziermacro{\on}
		%\reduce{\on}
		%\snapping{\off}
		%\pvinsert{% Your \input, \def, etc. here}
		%\quality{8.000}
		%\graddiff{0.005}
		%\snapasp{1}
		%\zoom{4.0000}
		\unitlength 0.65mm % = 2.845pt
		\linethickness{0.4pt}
		\ifx\plotpoint\undefined\newsavebox{\plotpoint}\fi % GNUPLOT compatibility
		\begin{picture}(64.375,45)(0,0)
		\put(39.165,21.87){\circle*{1.33}}
		\put(2.415,21.87){\circle*{1.33}}
		\put(54.165,21.87){\circle*{1.33}}
		\put(17.415,21.87){\circle*{1.33}}
		\put(54.165,34.87){\circle*{1.33}}
		\put(17.415,34.87){\circle*{1.33}}
		\put(54.165,34.87){\line(0,-1){13}}
		\put(17.415,34.87){\line(0,-1){13}}
		\put(54.165,21.87){\line(-1,0){15}}
		\put(17.415,21.87){\line(-1,0){15}}
		\qbezier(39.165,21.87)(39.665,22.5)(39.165,21.62)
		\qbezier(2.415,21.87)(2.915,22.5)(2.415,21.62)
		\qbezier(39.165,21.62)(39.045,22.62)(39.415,21.62)
		\qbezier(2.415,21.62)(2.295,22.62)(2.665,21.62)
		\qbezier(39.415,21.62)(38.545,22.87)(39.165,22.12)
		\qbezier(2.665,21.62)(1.795,22.87)(2.415,22.12)
		\put(38.665,34.37){\circle*{1.33}}
		\put(1.915,34.37){\circle*{1.33}}
		\put(38.665,34.87){\line(1,0){16}}
		\put(1.915,34.87){\line(1,0){16}}
		%\emline(54.665,34.87)(53.915,35.12)
		\multiput(54.665,34.87)(-.09375,.03125){8}{\line(-1,0){.09375}}
		%\end
		%\emline(17.915,34.87)(17.165,35.12)
		\multiput(17.915,34.87)(-.09375,.03125){8}{\line(-1,0){.09375}}
		%\end
		\put(53.915,35.12){\line(1,0){.25}}
		\put(17.165,35.12){\line(1,0){.25}}
		%\emline(17.165,22.12)(1.665,34.87)
		\multiput(17.165,22.12)(-.041005291,.0337301587){378}{\line(-1,0){.041005291}}
		%\end
		\put(38.665,34.87){\line(0,-1){12.75}}
		\put(1.915,34.87){\line(0,-1){12.75}}
		%\emline(17.165,35.12)(2.165,22.12)
		\multiput(17.165,35.12)(-.0388601036,-.0336787565){386}{\line(-1,0){.0388601036}}
		%\end
		\put(10.355,10.75){\makebox(0,0)[cc]{$H_{iv}$}}
		\qbezier(38.355,34.75)(43.73,45)(53.605,35.25)
		\qbezier(1.605,34.75)(6.98,45)(16.855,35.25)
		\qbezier(38.855,21.75)(46.23,12.625)(54.105,22)
		\qbezier(2.105,21.75)(9.48,12.625)(17.355,22)
		\put(28.5,3){\makebox(0,0)[cc]{Figure 5}}
		\bezier{542}(38.5,34.5)(28.75,27.375)(39,21.75)
		\bezier{555}(54.25,35)(64.375,27.5)(54,22)
		\put(45,11.25){\makebox(0,0)[cc]{$H_{v}$}}
		\end{picture}
		
	\end{center}

	However, $H_{v}$ contains a minor isomorphic to the graph $G_{7}$ while  $H_{iv}$  plus a loop contains  a minor isomorphic to $G_{6}$.  Hence $G$ does not arise from these graphs.
	
		\vskip.5cm
	\noindent \textbf{{Case (ii).}} Suppose  $G$ is   $2$-connected.

	By Harary [\cite{harary6graph}, pp. 230], every Eulerian graph with  $4$ vertices and $9$ edges is isomorphic to the graph $G_{13}.$ Hence $G$ is isomorphic to the graph $G_{13}.$  Thus $N/a \cong M(G_{13})$.
\end{proof}

\begin{lemma} \label{c2mfm*5}
	Let $N$ be a binary matroid and $a \in E(N)$ and $N$ does not contain a minor isomorphic to $M(G_6)$ or $M(G_7)$. If $N\backslash a \cong M^*(K_5)$ and $N/a$ is a graphic matroid, then $N/a\cong M(G_{14})$, where $G_{14}$ is the graph as shown in Figure 3.
\end{lemma}
\begin{proof} 
	
	Since $N/a$ is graphic, $N/a=M(G)$ for a connected graph $G$.
	As $M^*(K_5)$ has  rank $6$ and  $10$ elements,  $N$ has rank  $6$ and $11$ elements. Hence the graph $G$ has $6$ vertices and $10$ edges.
	By Lemma \ref{c2slemma}(i), $G$ is 2-connected or it is a block plus a loop and further, by Lemma \ref{c2slemma}(iv), $G$ does not contains a $2$-edge cut.

	\vskip.5cm\noindent
	{\bf Case (i).} Suppose  $G$ is a simple $2$-connected graph.
	
	By Harary [\cite{harary6graph}, pp. 223], there are $12$ non-isomorphic $2$-connected graphs with 6 vertices and 10 edges. Out of these 12 graphs, 8 graphs contain a vertex of degree $2$ each and so, by Lemma \ref{c2slemma} (iv),  we discard them. The remaining four graphs are as shown in Figure 6.

	\vskip.5cm
	\begin{center}
		% This is a LaTeX picture output by TeXCAD.
		% File name: [C2 forb dual of mk5.pic].
		% Version of TeXCAD: 4.3
		% Reference / build: 30-Jun-2012 (rev. 105)
		% For new versions, check: http://texcad.sf.net/
		% Options on the following lines.
		%\grade{\on}
		%\emlines{\off}
		%\epic{\off}
		%\beziermacro{\on}
		%\reduce{\on}
		%\snapping{\off}
		%\pvinsert{% Your \input, \def, etc. here}
		%\quality{8.000}
		%\graddiff{0.005}
		%\snapasp{1}
		%\zoom{4.0000}
		\unitlength 0.65mm % = 2.845pt
		\linethickness{0.4pt}
		\ifx\plotpoint\undefined\newsavebox{\plotpoint}\fi % GNUPLOT compatibility
		\begin{picture}(165.475,38.035)(0,0)
		\put(12.31,21.37){\circle*{1.33}}
		\put(55.31,21.12){\circle*{1.33}}
		\put(97.81,21.37){\circle*{1.33}}
		\put(139.31,21.37){\circle*{1.33}}
		\put(27.31,21.37){\circle*{1.33}}
		\put(70.31,21.12){\circle*{1.33}}
		\put(112.81,21.37){\circle*{1.33}}
		\put(154.31,21.37){\circle*{1.33}}
		\put(27.31,37.37){\circle*{1.33}}
		\put(70.31,37.12){\circle*{1.33}}
		\put(112.81,37.37){\circle*{1.33}}
		\put(154.31,37.37){\circle*{1.33}}
		\put(12.31,37.37){\circle*{1.33}}
		\put(55.31,37.12){\circle*{1.33}}
		\put(97.81,37.37){\circle*{1.33}}
		\put(139.31,37.37){\circle*{1.33}}
		\put(37.81,29.62){\circle*{1.33}}
		\put(80.81,29.37){\circle*{1.33}}
		\put(123.31,29.62){\circle*{1.33}}
		\put(164.81,29.62){\circle*{1.33}}
		\put(.06,29.87){\circle*{1.33}}
		\put(43.06,29.62){\circle*{1.33}}
		\put(85.56,29.87){\circle*{1.33}}
		\put(127.06,29.87){\circle*{1.33}}
		\put(27.31,21.37){\line(-1,0){15}}
		\put(70.31,21.12){\line(-1,0){15}}
		\put(112.81,21.37){\line(-1,0){15}}
		\put(154.31,21.37){\line(-1,0){15}}
		\qbezier(12.31,21.37)(12.81,22)(12.31,21.12)
		\qbezier(55.31,21.12)(55.81,21.75)(55.31,20.87)
		\qbezier(97.81,21.37)(98.31,22)(97.81,21.12)
		\qbezier(139.31,21.37)(139.81,22)(139.31,21.12)
		\qbezier(12.31,21.12)(12.19,22.12)(12.56,21.12)
		\qbezier(55.31,20.87)(55.19,21.87)(55.56,20.87)
		\qbezier(97.81,21.12)(97.69,22.12)(98.06,21.12)
		\qbezier(139.31,21.12)(139.19,22.12)(139.56,21.12)
		\qbezier(12.56,21.12)(11.69,22.37)(12.31,21.62)
		\qbezier(55.56,20.87)(54.69,22.12)(55.31,21.37)
		\qbezier(98.06,21.12)(97.19,22.37)(97.81,21.62)
		\qbezier(139.56,21.12)(138.69,22.37)(139.31,21.62)
		\put(83.5,-1.25){\makebox(0,0)[cc]{Figure 6}}
		\put(20.25,10.5){\makebox(0,0)[cc]{$H_{vi}$}}
		\put(63.25,10.25){\makebox(0,0)[cc]{$H_{vii}$}}
		\put(105.75,10.5){\makebox(0,0)[cc]{$H_{viii}$}}
		\put(147.25,10.5){\makebox(0,0)[cc]{$H_{ix}$}}
		%\emline(0,30)(12.25,37.25)
		\multiput(0,30)(.056976744,.03372093){215}{\line(1,0){.056976744}}
		%\end
		%\emline(43,29.75)(55.25,37)
		\multiput(43,29.75)(.056976744,.03372093){215}{\line(1,0){.056976744}}
		%\end
		%\emline(85.5,30)(97.75,37.25)
		\multiput(85.5,30)(.056976744,.03372093){215}{\line(1,0){.056976744}}
		%\end
		%\emline(127,30)(139.25,37.25)
		\multiput(127,30)(.056976744,.03372093){215}{\line(1,0){.056976744}}
		%\end
		%\emline(-.25,29.75)(12,21.25)
		\multiput(-.25,29.75)(.0486111111,-.0337301587){252}{\line(1,0){.0486111111}}
		%\end
		%\emline(42.75,29.5)(55,21)
		\multiput(42.75,29.5)(.0486111111,-.0337301587){252}{\line(1,0){.0486111111}}
		%\end
		%\emline(85.25,29.75)(97.5,21.25)
		\multiput(85.25,29.75)(.0486111111,-.0337301587){252}{\line(1,0){.0486111111}}
		%\end
		%\emline(126.75,29.75)(139,21.25)
		\multiput(126.75,29.75)(.0486111111,-.0337301587){252}{\line(1,0){.0486111111}}
		%\end
		\put(12.25,37.5){\line(1,0){14.75}}
		\put(55.25,37.25){\line(1,0){14.75}}
		\put(97.75,37.5){\line(1,0){14.75}}
		\put(139.25,37.5){\line(1,0){14.75}}
		%\emline(27,37.5)(37.75,29.5)
		\multiput(27,37.5)(.045168067,-.033613445){238}{\line(1,0){.045168067}}
		%\end
		%\emline(70,37.25)(80.75,29.25)
		\multiput(70,37.25)(.045168067,-.033613445){238}{\line(1,0){.045168067}}
		%\end
		%\emline(112.5,37.5)(123.25,29.5)
		\multiput(112.5,37.5)(.045168067,-.033613445){238}{\line(1,0){.045168067}}
		%\end
		%\emline(154,37.5)(164.75,29.5)
		\multiput(154,37.5)(.045168067,-.033613445){238}{\line(1,0){.045168067}}
		%\end
		%\emline(37.75,29.75)(27,21.25)
		\multiput(37.75,29.75)(-.0426587302,-.0337301587){252}{\line(-1,0){.0426587302}}
		%\end
		%\emline(80.75,29.5)(70,21)
		\multiput(80.75,29.5)(-.0426587302,-.0337301587){252}{\line(-1,0){.0426587302}}
		%\end
		%\emline(123.25,29.75)(112.5,21.25)
		\multiput(123.25,29.75)(-.0426587302,-.0337301587){252}{\line(-1,0){.0426587302}}
		%\end
		%\emline(164.75,29.75)(154,21.25)
		\multiput(164.75,29.75)(-.0426587302,-.0337301587){252}{\line(-1,0){.0426587302}}
		%\end
		%\emline(12.25,37.25)(27.25,21.25)
		\multiput(12.25,37.25)(.0337078652,-.0359550562){445}{\line(0,-1){.0359550562}}
		%\end
		%\emline(139.25,37.25)(154.25,21.25)
		\multiput(139.25,37.25)(.0337078652,-.0359550562){445}{\line(0,-1){.0359550562}}
		%\end
		\put(27,37.25){\line(0,-1){16}}
		\put(70,37){\line(0,-1){16}}
		\put(112.5,37.25){\line(0,-1){16}}
		\put(154,37.25){\line(0,-1){16}}
		%\emline(12.25,21)(37.5,29.5)
		\multiput(12.25,21)(.1001984127,.0337301587){252}{\line(1,0){.1001984127}}
		%\end
		%\emline(55.25,20.75)(80.5,29.25)
		\multiput(55.25,20.75)(.1001984127,.0337301587){252}{\line(1,0){.1001984127}}
		%\end
		%\emline(97.75,21)(123,29.5)
		\multiput(97.75,21)(.1001984127,.0337301587){252}{\line(1,0){.1001984127}}
		%\end
		%\emline(0,30.25)(27,21.5)
		\multiput(0,30.25)(.1038461538,-.0336538462){260}{\line(1,0){.1038461538}}
		%\end
		%\emline(43,30)(70,21.25)
		\multiput(43,30)(.1038461538,-.0336538462){260}{\line(1,0){.1038461538}}
		%\end
		\put(55,37.25){\line(0,-1){16.25}}
		\put(98.25,21.25){\line(-1,0){.5}}
		\put(97.75,37.25){\line(0,-1){16}}
		%\emline(85.25,29.75)(113,37.25)
		\multiput(85.25,29.75)(.124439462,.033632287){223}{\line(1,0){.124439462}}
		%\end
		%\emline(139,21.25)(154.25,37.5)
		\multiput(139,21.25)(.0337389381,.0359513274){452}{\line(0,1){.0359513274}}
		%\end
		%\emline(127,30)(164.75,29.75)
		\multiput(127,30)(4.71875,-.03125){8}{\line(1,0){4.71875}}
		%\end
		\end{picture}
		
	\end{center}

	The graphs $H_{vii}$ and $H_{viii}$  each contains a minor isomorphic to $G_{7}$. So we discard them also.   Suppose $G$ is isomorphic to $H_{ix}.$ Then the coextension $N$ of $M(G)$ is graphic or it contains an odd circuit that avoids $a.$ However, $N$ cannot be graphic as it contains a minor isomorphic to  $M^*(K_5).$  Hence $N$  contains an odd circuit without containing $a.$ Therefore  $N \backslash a \cong M^*(K_5)$ contains an odd circuit. This is  a contradiction to the fact that $M^*(K_5)$ is a bipartite matroid.   
	Hence $M(G)\ncong M(H_{ix})$. The only graph remained is $H_{vi}$. Hence $G\cong H_{vi}=G_{14}$.
	
	\vskip.5cm
	\noindent	
	{\bf Case (ii).} Suppose  $G$  contains either one pair of parallel edges or a loop.

	Then $G$ can be obtained from a simple graph on $6$ vertices and $9$ edges by adding a parallel edge or a loop.  By Harary [\cite{harary6graph}, pp. 222],  there are $14$ non-isomorphic $2$-connected graphs on $6$ vertices and $9$ edges.  Out of these, eight graphs contain more than two vertices of degree two giving a $2$-edge cut after adding a parallel edge or a loop. So we discard them by Lemma \ref{c2slemma}(iv).  The remaining graphs are as shown in Figure 7.
	\begin{center}
	% This is a LaTeX picture output by TeXCAD.
	% File name: [graphic quotient - Copy (3).pic].
	% Version of TeXCAD: 4.3
	% Reference / build: 30-Jun-2012 (rev. 105)
	% For new versions, check: http://texcad.sf.net/
	% Options on the following lines.
	%\grade{\on}
	%\emlines{\off}
	%\epic{\off}
	%\beziermacro{\on}
	%\reduce{\on}
	%\snapping{\off}
	%\pvinsert{% Your \input, \def, etc. here}
	%\quality{8.000}
	%\graddiff{0.005}
	%\snapasp{1}
	%\zoom{4.0000}
	\unitlength .65mm % = 1.992pt
	\linethickness{0.4pt}
	\ifx\plotpoint\undefined\newsavebox{\plotpoint}\fi % GNUPLOT compatibility
	\begin{picture}(241.475,33.035)(0,0)
	\put(134.06,15.87){\circle*{1.33}}
	\put(11.31,16.37){\circle*{1.33}}
	\put(174.56,15.62){\circle*{1.33}}
	\put(52.06,16.12){\circle*{1.33}}
	\put(215.31,15.12){\circle*{1.33}}
	\put(93.31,16.12){\circle*{1.33}}
	\put(149.06,15.87){\circle*{1.33}}
	\put(26.31,16.37){\circle*{1.33}}
	\put(189.56,15.62){\circle*{1.33}}
	\put(67.06,16.12){\circle*{1.33}}
	\put(230.31,15.12){\circle*{1.33}}
	\put(108.31,16.12){\circle*{1.33}}
	\put(149.06,31.87){\circle*{1.33}}
	\put(26.31,32.37){\circle*{1.33}}
	\put(189.56,31.62){\circle*{1.33}}
	\put(67.06,32.12){\circle*{1.33}}
	\put(230.31,31.12){\circle*{1.33}}
	\put(108.31,32.12){\circle*{1.33}}
	\put(134.06,31.87){\circle*{1.33}}
	\put(11.31,32.37){\circle*{1.33}}
	\put(174.56,31.62){\circle*{1.33}}
	\put(52.06,32.12){\circle*{1.33}}
	\put(215.31,31.12){\circle*{1.33}}
	\put(93.31,32.12){\circle*{1.33}}
	\put(159.56,24.12){\circle*{1.33}}
	\put(36.81,24.62){\circle*{1.33}}
	\put(200.06,23.87){\circle*{1.33}}
	\put(77.56,24.37){\circle*{1.33}}
	\put(240.81,23.37){\circle*{1.33}}
	\put(118.81,24.37){\circle*{1.33}}
	\put(121.81,24.37){\circle*{1.33}}
	\put(-.94,24.87){\circle*{1.33}}
	\put(162.31,24.12){\circle*{1.33}}
	\put(39.81,24.62){\circle*{1.33}}
	\put(203.06,23.62){\circle*{1.33}}
	\put(81.06,24.62){\circle*{1.33}}
	\put(149.06,15.87){\line(-1,0){15}}
	\put(26.31,16.37){\line(-1,0){15}}
	\put(189.56,15.62){\line(-1,0){15}}
	\put(67.06,16.12){\line(-1,0){15}}
	\put(230.31,15.12){\line(-1,0){15}}
	\put(108.31,16.12){\line(-1,0){15}}
	\qbezier(134.06,15.87)(134.56,16.5)(134.06,15.62)
	\qbezier(11.31,16.37)(11.81,17)(11.31,16.12)
	\qbezier(174.56,15.62)(175.06,16.25)(174.56,15.37)
	\qbezier(52.06,16.12)(52.56,16.75)(52.06,15.87)
	\qbezier(215.31,15.12)(215.81,15.75)(215.31,14.87)
	\qbezier(93.31,16.12)(93.81,16.75)(93.31,15.87)
	\qbezier(134.06,15.62)(133.94,16.62)(134.31,15.62)
	\qbezier(11.31,16.12)(11.19,17.12)(11.56,16.12)
	\qbezier(174.56,15.37)(174.44,16.37)(174.81,15.37)
	\qbezier(52.06,15.87)(51.94,16.87)(52.31,15.87)
	\qbezier(215.31,14.87)(215.19,15.87)(215.56,14.87)
	\qbezier(93.31,15.87)(93.19,16.87)(93.56,15.87)
	\qbezier(134.31,15.62)(133.44,16.87)(134.06,16.12)
	\qbezier(11.56,16.12)(10.69,17.37)(11.31,16.62)
	\qbezier(174.81,15.37)(173.94,16.62)(174.56,15.87)
	\qbezier(52.31,15.87)(51.44,17.12)(52.06,16.37)
	\qbezier(215.56,14.87)(214.69,16.12)(215.31,15.37)
	\qbezier(93.56,15.87)(92.69,17.12)(93.31,16.37)
	\put(128.75,-2){\makebox(0,0)[cc]{Figure 7}}
	\put(142,5){\makebox(0,0)[cc]{$H_{xiii}$}}
	\put(19.25,5.5){\makebox(0,0)[cc]{$H_{x}$}}
	\put(182.5,4.75){\makebox(0,0)[cc]{$H_{xiv}$}}
	\put(60,5.25){\makebox(0,0)[cc]{$H_{xi}$}}
	\put(223.25,4.25){\makebox(0,0)[cc]{$H_{xv}$}}
	\put(101.25,5.25){\makebox(0,0)[cc]{$H_{xii}$}}
	%\emline(121.75,24.5)(134,31.75)
	\multiput(121.75,24.5)(.081125828,.048013245){151}{\line(1,0){.081125828}}
	%\end
	%\emline(-1,25)(11.25,32.25)
	\multiput(-1,25)(.081125828,.048013245){151}{\line(1,0){.081125828}}
	%\end
	%\emline(162.25,24.25)(174.5,31.5)
	\multiput(162.25,24.25)(.081125828,.048013245){151}{\line(1,0){.081125828}}
	%\end
	%\emline(39.75,24.75)(52,32)
	\multiput(39.75,24.75)(.081125828,.048013245){151}{\line(1,0){.081125828}}
	%\end
	%\emline(203,23.75)(215.25,31)
	\multiput(203,23.75)(.081125828,.048013245){151}{\line(1,0){.081125828}}
	%\end
	%\emline(81,24.75)(93.25,32)
	\multiput(81,24.75)(.081125828,.048013245){151}{\line(1,0){.081125828}}
	%\end
	%\emline(121.5,24.25)(133.75,15.75)
	\multiput(121.5,24.25)(.06920904,-.048022599){177}{\line(1,0){.06920904}}
	%\end
	%\emline(-1.25,24.75)(11,16.25)
	\multiput(-1.25,24.75)(.06920904,-.048022599){177}{\line(1,0){.06920904}}
	%\end
	%\emline(162,24)(174.25,15.5)
	\multiput(162,24)(.06920904,-.048022599){177}{\line(1,0){.06920904}}
	%\end
	%\emline(39.5,24.5)(51.75,16)
	\multiput(39.5,24.5)(.06920904,-.048022599){177}{\line(1,0){.06920904}}
	%\end
	%\emline(202.75,23.5)(215,15)
	\multiput(202.75,23.5)(.06920904,-.048022599){177}{\line(1,0){.06920904}}
	%\end
	%\emline(80.75,24.5)(93,16)
	\multiput(80.75,24.5)(.06920904,-.048022599){177}{\line(1,0){.06920904}}
	%\end
	\put(134,32){\line(1,0){14.75}}
	\put(11.25,32.5){\line(1,0){14.75}}
	\put(174.5,31.75){\line(1,0){14.75}}
	\put(52,32.25){\line(1,0){14.75}}
	\put(215.25,31.25){\line(1,0){14.75}}
	\put(93.25,32.25){\line(1,0){14.75}}
	%\emline(148.75,32)(159.5,24)
	\multiput(148.75,32)(.064759036,-.048192771){166}{\line(1,0){.064759036}}
	%\end
	%\emline(26,32.5)(36.75,24.5)
	\multiput(26,32.5)(.064759036,-.048192771){166}{\line(1,0){.064759036}}
	%\end
	%\emline(189.25,31.75)(200,23.75)
	\multiput(189.25,31.75)(.064759036,-.048192771){166}{\line(1,0){.064759036}}
	%\end
	%\emline(66.75,32.25)(77.5,24.25)
	\multiput(66.75,32.25)(.064759036,-.048192771){166}{\line(1,0){.064759036}}
	%\end
	%\emline(230,31.25)(240.75,23.25)
	\multiput(230,31.25)(.064759036,-.048192771){166}{\line(1,0){.064759036}}
	%\end
	%\emline(108,32.25)(118.75,24.25)
	\multiput(108,32.25)(.064759036,-.048192771){166}{\line(1,0){.064759036}}
	%\end
	%\emline(159.5,24.25)(148.75,15.75)
	\multiput(159.5,24.25)(-.060734463,-.048022599){177}{\line(-1,0){.060734463}}
	%\end
	%\emline(36.75,24.75)(26,16.25)
	\multiput(36.75,24.75)(-.060734463,-.048022599){177}{\line(-1,0){.060734463}}
	%\end
	%\emline(200,24)(189.25,15.5)
	\multiput(200,24)(-.060734463,-.048022599){177}{\line(-1,0){.060734463}}
	%\end
	%\emline(77.5,24.5)(66.75,16)
	\multiput(77.5,24.5)(-.060734463,-.048022599){177}{\line(-1,0){.060734463}}
	%\end
	%\emline(240.75,23.5)(230,15)
	\multiput(240.75,23.5)(-.060734463,-.048022599){177}{\line(-1,0){.060734463}}
	%\end
	%\emline(118.75,24.5)(108,16)
	\multiput(118.75,24.5)(-.060734463,-.048022599){177}{\line(-1,0){.060734463}}
	%\end
	\put(26,32.25){\line(0,-1){16}}
	\put(108,32){\line(0,-1){16}}
	\put(174.25,31.75){\line(0,-1){16.25}}
	\put(51.75,32.25){\line(0,-1){16.25}}
	\put(215.75,15){\line(-1,0){.5}}
	\put(93.75,16){\line(-1,0){.5}}
	\put(93.25,32){\line(0,-1){16}}
	%\emline(80.75,24.5)(108.5,32)
	\multiput(80.75,24.5)(.177884615,.048076923){156}{\line(1,0){.177884615}}
	%\end
	\put(-1,24.5){\line(1,0){37.75}}
	\put(11,32.5){\line(0,-1){16}}
	%\emline(39.75,24.5)(67.25,31.75)
	\multiput(39.75,24.5)(.182119205,.048013245){151}{\line(1,0){.182119205}}
	%\end
	%\emline(51.75,32.5)(67.25,16)
	\multiput(51.75,32.5)(.048136646,-.051242236){322}{\line(0,-1){.051242236}}
	%\end
	\put(134,31.75){\line(0,-1){15.75}}
	\put(121.75,24.25){\line(1,0){37.75}}
	%\emline(134,31.75)(149,16)
	\multiput(134,31.75)(.0480769231,-.0504807692){312}{\line(0,-1){.0504807692}}
	%\end
	\put(162,23.75){\line(1,0){38}}
	\put(200,23.75){\line(0,-1){.25}}
	%\emline(200,23.5)(174.25,15.75)
	\multiput(200,23.5)(-.159937888,-.048136646){161}{\line(-1,0){.159937888}}
	%\end
	\put(203,23.5){\line(1,0){38}}
	%\emline(215,15.25)(230.25,31.25)
	\multiput(215,15.25)(.0481072555,.0504731861){317}{\line(0,1){.0504731861}}
	%\end
	%\emline(215.25,31.5)(230,15.5)
	\multiput(215.25,31.5)(.0480456026,-.0521172638){307}{\line(0,-1){.0521172638}}
	%\end
	\end{picture}
	\end{center}

	Out of these, $H_x$ and $H_{xii}$ each contains the graph $G_{7}$ as a minor. Hence, they can be discarded.  Each of the remaining four graphs contains $G_{6}$  as a minor after addition of a loop. Similarly,  each of these graphs  contains  $G_{7}$ as a minor after the addition of one parallel edge in such a way that resulting graph does not contain a $2$-edge cut. Thus, in this case, all choices for $G$ are discarded.  
	
	\vskip.5cm
	\noindent
	{\bf Case (iii).} Suppose  $G$  contains either two  pairs of parallel edges, or  one pair of parallel edges and a  loop.

	Then $G$ can be obtained from a simple graph on $6$ vertices and $8$ edges.  There are nine non-isomorphic $2$-connected graphs on $6$ vertices and $8$ edges by Harary [\cite{harary6graph}, pp. 221]. Out of these, five contain more than three vertices  of degree two giving a 2-edge cut in $G$ after adding two parallel edges or one parallel edge and a loop. Hence  we discard them by Lemma \ref{c2slemma}(iv). The remaining graphs are 
	$H_{xvi}$, $H_{xvii}$, $H_{xviii}$ and $H_{xix}$ as shown in Figure 8.

	\begin{center} 
		% This is a LaTeX picture output by TeXCAD.
		% File name: [fig 10.pic].
		% Version of TeXCAD: 4.3
		% Reference / build: 30-Jun-2012 (rev. 105)
		% For new versions, check: http://texcad.sf.net/
		% Options on the following lines.
		%\grade{\on}
		%\emlines{\off}
		%\epic{\off}
		%\beziermacro{\on}
		%\reduce{\on}
		%\snapping{\off}
		%\pvinsert{% Your \input, \def, etc. here}
		%\quality{8.000}
		%\graddiff{0.005}
		%\snapasp{1}
		%\zoom{4.0000}
		\unitlength .65mm % = 1.849pt
		\linethickness{0.4pt}
		\ifx\plotpoint\undefined\newsavebox{\plotpoint}\fi % GNUPLOT compatibility
		\begin{picture}(167.225,27.535)(0,0)
		\put(12.31,10.87){\circle*{1.33}}
		\put(55.31,10.62){\circle*{1.33}}
		\put(97.81,10.87){\circle*{1.33}}
		\put(141.06,10.62){\circle*{1.33}}
		\put(27.31,10.87){\circle*{1.33}}
		\put(70.31,10.62){\circle*{1.33}}
		\put(112.81,10.87){\circle*{1.33}}
		\put(156.06,10.62){\circle*{1.33}}
		\put(27.31,26.87){\circle*{1.33}}
		\put(70.31,26.62){\circle*{1.33}}
		\put(112.81,26.87){\circle*{1.33}}
		\put(156.06,26.62){\circle*{1.33}}
		\put(12.31,26.87){\circle*{1.33}}
		\put(55.31,26.62){\circle*{1.33}}
		\put(97.81,26.87){\circle*{1.33}}
		\put(141.06,26.62){\circle*{1.33}}
		\put(37.81,19.12){\circle*{1.33}}
		\put(80.81,18.87){\circle*{1.33}}
		\put(123.31,19.12){\circle*{1.33}}
		\put(166.56,18.87){\circle*{1.33}}
		\put(.06,19.37){\circle*{1.33}}
		\put(43.06,19.12){\circle*{1.33}}
		\put(85.56,19.37){\circle*{1.33}}
		\put(128.81,19.12){\circle*{1.33}}
		\put(27.31,10.87){\line(-1,0){15}}
		\put(70.31,10.62){\line(-1,0){15}}
		\put(112.81,10.87){\line(-1,0){15}}
		\put(156.06,10.62){\line(-1,0){15}}
		\qbezier(12.31,10.87)(12.81,11.5)(12.31,10.62)
		\qbezier(55.31,10.62)(55.81,11.25)(55.31,10.37)
		\qbezier(97.81,10.87)(98.31,11.5)(97.81,10.62)
		\qbezier(141.06,10.62)(141.56,11.25)(141.06,10.37)
		\qbezier(12.31,10.62)(12.19,11.62)(12.56,10.62)
		\qbezier(55.31,10.37)(55.19,11.37)(55.56,10.37)
		\qbezier(97.81,10.62)(97.69,11.62)(98.06,10.62)
		\qbezier(141.06,10.37)(140.94,11.37)(141.31,10.37)
		\qbezier(12.56,10.62)(11.69,11.87)(12.31,11.12)
		\qbezier(55.56,10.37)(54.69,11.62)(55.31,10.87)
		\qbezier(98.06,10.62)(97.19,11.87)(97.81,11.12)
		\qbezier(141.31,10.37)(140.44,11.62)(141.06,10.87)
		\put(80.75,-8){\makebox(0,0)[cc]{Figure 8}}
		\put(20.25,1){\makebox(0,0)[cc]{$H_{xvi}$}}
		\put(63.25,1){\makebox(0,0)[cc]{$H_{xvii}$}}
		\put(105.75,1){\makebox(0,0)[cc]{$H_{xviii}$}}
		\put(149,1){\makebox(0,0)[cc]{$H_{xix}$}}
		%\emline(0,19.5)(12.25,26.75)
		\multiput(0,19.5)(.0875,.051785714){140}{\line(1,0){.0875}}
		%\end
		%\emline(43,19.25)(55.25,26.5)
		\multiput(43,19.25)(.0875,.051785714){140}{\line(1,0){.0875}}
		%\end
		%\emline(85.5,19.5)(97.75,26.75)
		\multiput(85.5,19.5)(.0875,.051785714){140}{\line(1,0){.0875}}
		%\end
		%\emline(128.75,19.25)(141,26.5)
		\multiput(128.75,19.25)(.0875,.051785714){140}{\line(1,0){.0875}}
		%\end
		%\emline(-.25,19.25)(12,10.75)
		\multiput(-.25,19.25)(.074695122,-.051829268){164}{\line(1,0){.074695122}}
		%\end
		%\emline(42.75,19)(55,10.5)
		\multiput(42.75,19)(.074695122,-.051829268){164}{\line(1,0){.074695122}}
		%\end
		%\emline(85.25,19.25)(97.5,10.75)
		\multiput(85.25,19.25)(.074695122,-.051829268){164}{\line(1,0){.074695122}}
		%\end
		%\emline(128.5,19)(140.75,10.5)
		\multiput(128.5,19)(.074695122,-.051829268){164}{\line(1,0){.074695122}}
		%\end
		\put(12.25,27){\line(1,0){14.75}}
		\put(55.25,26.75){\line(1,0){14.75}}
		\put(97.75,27){\line(1,0){14.75}}
		\put(141,26.75){\line(1,0){14.75}}
		%\emline(27,27)(37.75,19)
		\multiput(27,27)(.069354839,-.051612903){155}{\line(1,0){.069354839}}
		%\end
		%\emline(70,26.75)(80.75,18.75)
		\multiput(70,26.75)(.069354839,-.051612903){155}{\line(1,0){.069354839}}
		%\end
		%\emline(112.5,27)(123.25,19)
		\multiput(112.5,27)(.069354839,-.051612903){155}{\line(1,0){.069354839}}
		%\end
		%\emline(155.75,26.75)(166.5,18.75)
		\multiput(155.75,26.75)(.069354839,-.051612903){155}{\line(1,0){.069354839}}
		%\end
		%\emline(37.75,19.25)(27,10.75)
		\multiput(37.75,19.25)(-.06554878,-.051829268){164}{\line(-1,0){.06554878}}
		%\end
		%\emline(80.75,19)(70,10.5)
		\multiput(80.75,19)(-.06554878,-.051829268){164}{\line(-1,0){.06554878}}
		%\end
		%\emline(123.25,19.25)(112.5,10.75)
		\multiput(123.25,19.25)(-.06554878,-.051829268){164}{\line(-1,0){.06554878}}
		%\end
		%\emline(166.5,19)(155.75,10.5)
		\multiput(166.5,19)(-.06554878,-.051829268){164}{\line(-1,0){.06554878}}
		%\end
		\put(27,26.75){\line(0,-1){16}}
		\put(55,26.75){\line(0,-1){16.25}}
		\put(98.25,10.75){\line(-1,0){.5}}
		\put(141.5,10.5){\line(-1,0){.5}}
		\put(97.75,26.75){\line(0,-1){16}}
		%\emline(85.25,19.25)(113,26.75)
		\multiput(85.25,19.25)(.19137931,.051724138){145}{\line(1,0){.19137931}}
		%\end
		\put(12,27){\line(0,-1){16}}
		\put(42.75,18.75){\line(1,0){37.75}}
		%\emline(140.75,26.75)(156,10.5)
		\multiput(140.75,26.75)(.0518707483,-.0552721088){294}{\line(0,-1){.0552721088}}
		%\end
		%\emline(140.75,10.75)(156.25,26.75)
		\multiput(140.75,10.75)(.0518394649,.0535117057){299}{\line(0,1){.0535117057}}
		%\end
		\end{picture}
		
	\end{center}
	
	\vskip.4cm

	In each of these graphs, addition of two parallel edges gives a minor  isomorphic to $G_{7}$ or contains a  $2$-edge cut.  Similarly, after addition of one parallel edge and a loop to each of these graphs, the resulting graph contains $G_{7}$ as a minor or contains a $2$-edge cut.  Thus, in this case, all choices for $G$ get discarded.

	\vskip.5cm\noindent
	{\bf Case (iv).} Suppose  $G$  contains three  pairs of parallel edges, or two pairs of parallel edges and a loop.

	Then $G$ can be obtained from simple $2$-connected graph with $6$ vertices and $7$ edges. Then only $H_{xx}$, $H_{xxi}$ and $H_{xxii}$ are the simple, non-isomorphic $2$-connected graphs [\cite{harary6graph}, pp. 220] on $6$ vertices and $7$ edges as shown in Figure 9. 
	
	\begin{center}
		% This is a LaTeX picture output by TeXCAD.
		% File name: [fig 11.pic].
		% Version of TeXCAD: 4.3
		% Reference / build: 30-Jun-2012 (rev. 105)
		% For new versions, check: http://texcad.sf.net/
		% Options on the following lines.
		%\grade{\on}
		%\emlines{\off}
		%\epic{\off}
		%\beziermacro{\on}
		%\reduce{\on}
		%\snapping{\off}
		%\pvinsert{% Your \input, \def, etc. here}
		%\quality{8.000}
		%\graddiff{0.005}
		%\snapasp{1}
		%\zoom{4.0000}
		\unitlength .65mm % = 1.849pt
		\linethickness{0.4pt}
		\ifx\plotpoint\undefined\newsavebox{\plotpoint}\fi % GNUPLOT compatibility
		\begin{picture}(125.08,35.535)(0,0)
		\put(13.415,18.87){\circle*{1.33}}
		\put(56.415,18.62){\circle*{1.33}}
		\put(98.915,18.87){\circle*{1.33}}
		\put(28.415,18.87){\circle*{1.33}}
		\put(71.415,18.62){\circle*{1.33}}
		\put(113.915,18.87){\circle*{1.33}}
		\put(28.415,34.87){\circle*{1.33}}
		\put(71.415,34.62){\circle*{1.33}}
		\put(113.915,34.87){\circle*{1.33}}
		\put(13.415,34.87){\circle*{1.33}}
		\put(56.415,34.62){\circle*{1.33}}
		\put(98.915,34.87){\circle*{1.33}}
		\put(38.915,27.12){\circle*{1.33}}
		\put(81.915,26.87){\circle*{1.33}}
		\put(124.415,27.12){\circle*{1.33}}
		\put(1.165,27.37){\circle*{1.33}}
		\put(44.165,27.12){\circle*{1.33}}
		\put(86.665,27.37){\circle*{1.33}}
		\put(28.415,18.87){\line(-1,0){15}}
		\put(71.415,18.62){\line(-1,0){15}}
		\put(113.915,18.87){\line(-1,0){15}}
		\qbezier(13.415,18.87)(13.915,19.5)(13.415,18.62)
		\qbezier(56.415,18.62)(56.915,19.25)(56.415,18.37)
		\qbezier(98.915,18.87)(99.415,19.5)(98.915,18.62)
		\qbezier(13.415,18.62)(13.295,19.62)(13.665,18.62)
		\qbezier(56.415,18.37)(56.295,19.37)(56.665,18.37)
		\qbezier(98.915,18.62)(98.795,19.62)(99.165,18.62)
		\qbezier(13.665,18.62)(12.795,19.87)(13.415,19.12)
		\qbezier(56.665,18.37)(55.795,19.62)(56.415,18.87)
		\qbezier(99.165,18.62)(98.295,19.87)(98.915,19.12)
		\put(21.355,8){\makebox(0,0)[cc]{$H_{xx}$}}
		\put(64.355,7.75){\makebox(0,0)[cc]{$H_{xxi}$}}
		\put(106.855,8){\makebox(0,0)[cc]{$H_{xxii}$}}
		%\emline(1.105,27.5)(13.355,34.75)
		\multiput(1.105,27.5)(.056976744,.03372093){215}{\line(1,0){.056976744}}
		%\end
		%\emline(44.105,27.25)(56.355,34.5)
		\multiput(44.105,27.25)(.056976744,.03372093){215}{\line(1,0){.056976744}}
		%\end
		%\emline(86.605,27.5)(98.855,34.75)
		\multiput(86.605,27.5)(.056976744,.03372093){215}{\line(1,0){.056976744}}
		%\end
		%\emline(.855,27.25)(13.105,18.75)
		\multiput(.855,27.25)(.0486111111,-.0337301587){252}{\line(1,0){.0486111111}}
		%\end
		%\emline(43.855,27)(56.105,18.5)
		\multiput(43.855,27)(.0486111111,-.0337301587){252}{\line(1,0){.0486111111}}
		%\end
		%\emline(86.355,27.25)(98.605,18.75)
		\multiput(86.355,27.25)(.0486111111,-.0337301587){252}{\line(1,0){.0486111111}}
		%\end
		\put(13.355,35){\line(1,0){14.75}}
		\put(56.355,34.75){\line(1,0){14.75}}
		%\emline(28.105,35)(38.855,27)
		\multiput(28.105,35)(.045168067,-.033613445){238}{\line(1,0){.045168067}}
		%\end
		%\emline(71.105,34.75)(81.855,26.75)
		\multiput(71.105,34.75)(.045168067,-.033613445){238}{\line(1,0){.045168067}}
		%\end
		%\emline(113.605,35)(124.355,27)
		\multiput(113.605,35)(.045168067,-.033613445){238}{\line(1,0){.045168067}}
		%\end
		%\emline(38.855,27.25)(28.105,18.75)
		\multiput(38.855,27.25)(-.0426587302,-.0337301587){252}{\line(-1,0){.0426587302}}
		%\end
		%\emline(81.855,27)(71.105,18.5)
		\multiput(81.855,27)(-.0426587302,-.0337301587){252}{\line(-1,0){.0426587302}}
		%\end
		%\emline(124.355,27.25)(113.605,18.75)
		\multiput(124.355,27.25)(-.0426587302,-.0337301587){252}{\line(-1,0){.0426587302}}
		%\end
		\put(56.105,34.75){\line(0,-1){16.25}}
		\put(99.355,18.75){\line(-1,0){.5}}
		\put(1.105,27.25){\line(1,0){37.75}}
		\put(81.855,26.75){\line(0,-1){.25}}
		%\emline(86.605,27.25)(114.105,34.75)
		\multiput(86.605,27.25)(.123318386,.033632287){223}{\line(1,0){.123318386}}
		%\end
		%\emline(98.855,35)(124.355,27)
		\multiput(98.855,35)(.107142857,-.033613445){238}{\line(1,0){.107142857}}
		%\end
		\put(62.5,.75){\makebox(0,0)[cc]{Figure 9}}
		\end{picture}
		
	\end{center}

	  For each of these graphs, if we add three parallel edges, or two parallel edges and one loop,  then the resulting graphs contain  $G_{7}$ as a minor or a $2$-edge cut.   Hence $G$ has no choice in this case also. 
	
	There is no $2$-connected graph without containing a minor $G_{7}$ or a $2$-edge cut with  at least three pairs of parallel edges and a loop, or at least four pairs of parallel edges. 	
\end{proof}

\section{Proof of Main Theorem}
In this section, we prove  Theorem \ref{c2grspl3} using the results of the last two sections.  We use the following result  and the fact that, a matroid $M$ is  {\it  Eulerian} if its ground set is an union of disjoint circuits of $M.$  

\begin{lemma}\cite{shikare2011generalized}\label{c18} Let $M$ be a  binary matroid and $T \subseteq E(M)$ and $\mathcal{C}$ be the collection of circuits of $M$. Then  the collection of circuits of $M_T$ is $ \mathcal{C}_1 \cup  \mathcal{C}_2$, where \\
	$\mathcal{C}_1=\{C \in  \mathcal{C}~\colon ~|C \cap T| ~is ~even\}$ and \\
	$\mathcal{C}_2=\{C_1 \cup C_2~\colon~C_1, C_2\in \mathcal{C},~C_1 \cap C_2= \emptyset,~ |C_i\cap T|~is~odd,~C_1 \cup C_2~contains~no~member ~of ~ \mathcal{C}_1 \}.$  
	
\end{lemma}

We restate Theorem \ref{c2grspl3} here for convenience. 

\begin{theorem} 
	Let $M$ be a graphic matroid. Then the splitting matroid $M_T$ is graphic for any $T\subseteq E(M)$ with $|T|=3$ if and only if  $M$  does not contain a minor isomorphic to any of the circuit matroids  $M(\tilde{G_4}),  M(\tilde{G_5}) $ and $M(G_k)$ for $k=6,7,8,9,$ where  $G_4, G_5, \dots, G_{9}$ are the graphs as shown in Figure 2.
\end{theorem} 

\begin{proof}   Suppose $M$ has a minor  isomorphic to one of the circuit matroids $M(\tilde{G_4}),  M(\tilde{G_5}) $ and $M(G_k)$ for $k= 6,7,8,9.$ Then, by Proposition \ref{c2exten2}, $M_T$ is not graphic for some $T \subseteq E(M)$ with $|T| = 3.$

	Conversely, suppose  that $M$ does not contain a minor  isomorphic to any of the circuits matroids $M(\tilde{G_4}),  M(\tilde{G_5}) $ and $M(G_k)$ for $k= 6,7,8,9.$  We prove that $M_T$ is graphic for any $T\subseteq E(M)$ with $|T|=3$.
	On the contrary, assume that $M_T$ is not a graphic matroid for some $ T \subseteq E(M)$ with $|T| = 3.$ Then, by Theorem \ref{c1cgm}, $M_T$  has a minor isomorphic to  $F$ for some  $F \in \mathcal{F} = \{ F_7, F_7^*, M^*(K_{3,3}), M^*(K_5)\}$. 
	
	By Lemma \ref{c2minle}, $M$ contains a minor $Q$ containing $T$ such that one of the following holds. 
	\begin{enumerate}[label=(\roman*).]
		\item $Q$ is isomorphic to one of the circuits matroids $M(\tilde{G_1}),  M(\tilde{G_2})$ or $M(\tilde{G_3})$, where $G_i$ is the graph as shown in Figure 1 for $i=1,2,3.$  
		\item $ Q_T \cong F $ or $Q_T/T' \cong F$  for some non-empty subset $ T'$ of $T.$  
	\end{enumerate}

	Suppose (i) holds.  Since  $G_1=G_5$ and $G_3=G_4$, $Q$ is isomorphic to $M(\tilde{G_2}).$  Note that $G_7$ can be obtained by contracting one edge of $G_2$.  Hence $M(G_{7})$ is a minor of $M(\tilde{G_2})$. Thus $M$ has a minor isomorphic to $M(G_{7})$, a contradiction.  	Therefore (i) does not hold.

	Suppose (ii) holds. By Lemmas \ref{c2nocckts} and  \ref{c2nocirut}, we may assume that $Q$ does not contain a 2-cocircuit and a coloop. Hence it is minimal with respect to $F$. 
	By Lemma \ref{c2mml}, there exists a binary matroid  $N$ such that $N \backslash a \cong F$, and $Q=N/a$ or $Q$ is a coextension of $N/a$ by one, two or three elements. 
	As $M$ is graphic, $Q$ and $N/a$ are also graphic. Let $Q=M(G)$ for some graph $G$. We may assume that $G$ is connected.  There are the following four cases depending on the choice of $F.$  We obtain a contradiction in each case using Lemma \ref{c2mff*7}, \ref{c2mff7}, \ref{c2mfm*33} and \ref{c2mfm*5}. 
	\vskip.3cm
	\noindent 
	{\bf  Case (i)}.  Suppose   $F\cong F^*_7$. 
	
	Since $N\backslash a \cong F$,  we have $ N\backslash a \cong F^*_7.$ By Lemma \ref{c2mff*7}, $N/a$ is isomorphic to $M(G_{10})$ or $M(G_{11})$, where $G_{10}$ or $G_{11}$ are the graphs as shown in Figure 3.  Note that  $G_{10}=G_{6}$ and $G_{11}=G_{7}.$  Therefore $M(G_6)$ or $M(G_7)$ is isomorphic to a minor of $M$, a contradiction.

	\vskip.3cm
	\noindent  
	{\bf  Case (ii).} Suppose  $F \cong F_7$.

	Then $ N\backslash a \cong F_7.$ By Lemma \ref{c2mff7}, $N/a$  is isomorphic to $M(G_{12})$, where $G_{12}$ is the graph as shown in Figure 3.  Since $G_{12}$ contains three pairs of parallel edges and a loop, for  any $T \subseteq E(G_{12})$ with $|T|=3$,  $M(G_{12})_T $ contains a $2$-circuit  or  a loop.  Therefore $M(G_{12})_T \ncong F_7$.  Hence $Q\ncong M(G_{12})$. Therefore $Q$ is a coextension of $N/a=M(G_{12})$ by one, two or three elements.

	There are five  non-isomorphic coextensions  $G_{15}$, $G_{16}$, $G_{17}$, $G_{18}$ and $G_{19}$ of $G_{12}$ by an edge without containing  a 2-edge cut and a coloop as shown in Figure 10.

	\begin{center}
		% This is a LaTeX picture output by TeXCAD.
		% File name: [graphic quotient - Copy (3) - Copy.pic].
		% Version of TeXCAD: 4.3
		% Reference / build: 30-Jun-2012 (rev. 105)
		% For new versions, check: http://texcad.sf.net/
		% Options on the following lines.
		%\grade{\on}
		%\emlines{\off}
		%\epic{\off}
		%\beziermacro{\on}
		%\reduce{\on}
		%\snapping{\off}
		%\pvinsert{% Your \input, \def, etc. here}
		%\quality{8.000}
		%\graddiff{0.005}
		%\snapasp{1}
		%\zoom{4.0000}
		\unitlength .7mm % = 1.992pt
		\linethickness{0.4pt}
		\ifx\plotpoint\undefined\newsavebox{\plotpoint}\fi % GNUPLOT compatibility
		\begin{picture}(181.025,45)(0,0)
		\put(5.56,19.12){\circle*{1.33}}
		\put(89.81,19.12){\circle*{1.33}}
		\put(114.31,19.12){\circle*{1.33}}
		\put(34.81,18.62){\circle*{1.33}}
		\put(64.31,18.12){\circle*{1.33}}
		\put(20.56,19.12){\circle*{1.33}}
		\put(104.81,19.12){\circle*{1.33}}
		\put(129.31,19.12){\circle*{1.33}}
		\put(49.81,18.62){\circle*{1.33}}
		\put(79.31,18.12){\circle*{1.33}}
		\put(20.56,32.12){\circle*{1.33}}
		\put(104.81,32.12){\circle*{1.33}}
		\put(129.31,32.12){\circle*{1.33}}
		\put(49.81,31.62){\circle*{1.33}}
		\put(79.31,31.12){\circle*{1.33}}
		\put(20.56,32.12){\line(0,-1){13}}
		\put(104.81,32.12){\line(0,-1){13}}
		\put(129.31,32.12){\line(0,-1){13}}
		\put(49.81,31.62){\line(0,-1){13}}
		\put(79.31,31.12){\line(0,-1){13}}
		\put(20.56,19.12){\line(-1,0){15}}
		\put(104.81,19.12){\line(-1,0){15}}
		\put(129.31,19.12){\line(-1,0){15}}
		\put(49.81,18.62){\line(-1,0){15}}
		\put(79.31,18.12){\line(-1,0){15}}
		\qbezier(5.56,19.12)(6.06,19.75)(5.56,18.87)
		\qbezier(89.81,19.12)(90.31,19.75)(89.81,18.87)
		\qbezier(114.31,19.12)(114.81,19.75)(114.31,18.87)
		\qbezier(34.81,18.62)(35.31,19.25)(34.81,18.37)
		\qbezier(64.31,18.12)(64.81,18.75)(64.31,17.87)
		\qbezier(5.56,18.87)(5.44,19.87)(5.81,18.87)
		\qbezier(89.81,18.87)(89.69,19.87)(90.06,18.87)
		\qbezier(114.31,18.87)(114.19,19.87)(114.56,18.87)
		\qbezier(34.81,18.37)(34.69,19.37)(35.06,18.37)
		\qbezier(64.31,17.87)(64.19,18.87)(64.56,17.87)
		\qbezier(5.81,18.87)(4.94,20.12)(5.56,19.37)
		\qbezier(90.06,18.87)(89.19,20.12)(89.81,19.37)
		\qbezier(114.56,18.87)(113.69,20.12)(114.31,19.37)
		\qbezier(35.06,18.37)(34.19,19.62)(34.81,18.87)
		\qbezier(64.56,17.87)(63.69,19.12)(64.31,18.37)
		\put(5.06,31.62){\circle*{1.33}}
		\put(89.31,31.62){\circle*{1.33}}
		\put(113.81,31.62){\circle*{1.33}}
		\put(34.31,31.12){\circle*{1.33}}
		\put(63.81,30.62){\circle*{1.33}}
		\put(5.06,32.12){\line(1,0){16}}
		\put(89.31,32.12){\line(1,0){16}}
		\put(113.81,32.12){\line(1,0){16}}
		\put(34.31,31.62){\line(1,0){16}}
		\put(63.81,31.12){\line(1,0){16}}
		%\emline(21.06,32.12)(20.31,32.37)
		\multiput(21.06,32.12)(-.125,.041667){6}{\line(-1,0){.125}}
		%\end
		%\emline(105.31,32.12)(104.56,32.37)
		\multiput(105.31,32.12)(-.125,.041667){6}{\line(-1,0){.125}}
		%\end
		%\emline(129.81,32.12)(129.06,32.37)
		\multiput(129.81,32.12)(-.125,.041667){6}{\line(-1,0){.125}}
		%\end
		%\emline(50.31,31.62)(49.56,31.87)
		\multiput(50.31,31.62)(-.125,.041667){6}{\line(-1,0){.125}}
		%\end
		%\emline(79.81,31.12)(79.06,31.37)
		\multiput(79.81,31.12)(-.125,.041667){6}{\line(-1,0){.125}}
		%\end
		\put(20.31,32.37){\line(1,0){.25}}
		\put(104.56,32.37){\line(1,0){.25}}
		\put(129.06,32.37){\line(1,0){.25}}
		\put(49.56,31.87){\line(1,0){.25}}
		\put(79.06,31.37){\line(1,0){.25}}
		%\emline(49.56,18.87)(34.06,31.62)
		\multiput(49.56,18.87)(-.058490566,.0481132075){265}{\line(-1,0){.058490566}}
		%\end
		%\emline(79.06,18.37)(63.56,31.12)
		\multiput(79.06,18.37)(-.058490566,.0481132075){265}{\line(-1,0){.058490566}}
		%\end
		\put(5.06,32.12){\line(0,-1){12.75}}
		\put(89.31,32.12){\line(0,-1){12.75}}
		\put(113.81,32.12){\line(0,-1){12.75}}
		\put(34.31,31.62){\line(0,-1){12.75}}
		\put(63.81,31.12){\line(0,-1){12.75}}
		%\emline(49.56,31.87)(34.56,18.87)
		\multiput(49.56,31.87)(-.0555555556,-.0481481481){270}{\line(-1,0){.0555555556}}
		%\end
		%\emline(79.06,31.37)(64.06,18.37)
		\multiput(79.06,31.37)(-.0555555556,-.0481481481){270}{\line(-1,0){.0555555556}}
		%\end
		\put(92,-3){\makebox(0,0)[cc]{Figure 10}}
		\put(14,8){\makebox(0,0)[cc]{$G_{15}$}}
		\put(98.25,8){\makebox(0,0)[cc]{$G_{18}$}}
		\put(122.75,8){\makebox(0,0)[cc]{$G_{19}$}}
		\put(43.25,7.5){\makebox(0,0)[cc]{$G_{16}$}}
		\put(72.75,7){\makebox(0,0)[cc]{$G_{17}$}}
		\bezier{579}(5,32)(-5.75,24.75)(5.5,19.5)
		\bezier{579}(89.25,32)(78.5,24.75)(89.75,19.5)
		\qbezier(20.25,19.75)(21.25,7.625)(29.25,13)
		\qbezier(49.75,19.25)(50.75,7.125)(58.75,12.5)
		\qbezier(29.25,12.5)(29.75,20.5)(20.25,19.5)
		\qbezier(58.75,12)(59.25,20)(49.75,19)
		\qbezier(20.5,32.25)(29.625,22.75)(20.25,19.25)
		\qbezier(104.75,32.25)(113.875,22.75)(104.5,19.25)
		\qbezier(129.25,32.25)(138.375,22.75)(129,19.25)
		\qbezier(5.25,19)(12.125,7.125)(20.5,18.75)
		\qbezier(89.5,19)(96.375,7.125)(104.75,18.75)
		\qbezier(114,19)(120.875,7.125)(129.25,18.75)
		\qbezier(34.5,18.5)(42.875,7.375)(49.75,18.75)
		\qbezier(64,18)(72.375,6.875)(79.25,18.25)
		\qbezier(63.75,30.5)(71.125,43.125)(79,31.25)
		\qbezier(89.25,32)(96.5,45)(104.75,32)
		\qbezier(113.75,32)(121,45)(129.25,32)
		%\emline(114,19.25)(129.25,32)
		\multiput(114,19.25)(.0575471698,.0481132075){265}{\line(1,0){.0575471698}}
		%\end
		\put(137.665,31.87){\circle*{1.33}}
		\put(160.665,31.62){\circle*{1.33}}
		\put(154.005,31.87){\circle*{1.33}}
		\put(177.005,31.62){\circle*{1.33}}
		\put(154.005,18.87){\circle*{1.33}}
		\put(177.005,18.62){\circle*{1.33}}
		\put(137.665,18.87){\circle*{1.33}}
		\put(160.665,18.62){\circle*{1.33}}
		\put(145.665,42.87){\circle*{1.33}}
		\put(168.665,42.62){\circle*{1.33}}
		\put(137.665,31.87){\line(1,0){16.33}}
		\put(160.665,31.62){\line(1,0){16.33}}
		\put(154.005,31.87){\line(0,-1){13}}
		\put(177.005,31.62){\line(0,-1){13}}
		\put(154.005,18.87){\line(-1,0){16.33}}
		\put(177.005,18.62){\line(-1,0){16.33}}
		\put(137.665,18.87){\line(0,1){13}}
		\put(160.665,18.62){\line(0,1){13}}
		\put(137.665,31.87){\line(3,4){8.33}}
		\put(160.665,31.62){\line(3,4){8.33}}
		\put(145.665,42.87){\line(3,-4){8.33}}
		\put(168.665,42.62){\line(3,-4){8.33}}
		\put(154.005,31.87){\line(-5,-4){16.33}}
		\put(154.005,18.87){\line(-5,4){16.33}}
		\put(145.665,42.87){\line(-1,-3){8}}
		\put(168.665,42.62){\line(-1,-3){8}}
		\put(145.665,42.87){\line(1,-3){8}}
		\put(168.665,42.62){\line(1,-3){8}}
		\put(145.275,7.75){\makebox(0,0)[cc]{$G_{20}$}}
		\put(168.275,7.5){\makebox(0,0)[cc]{$G_{21}$}}
		\qbezier(168.775,42.5)(181.025,41)(177.275,31.5)
		\qbezier(168.775,42.5)(157.4,41.375)(160.525,31.75)
		\end{picture}
			
	\end{center}
	
	Since $G_{15}$, $G_{18}$ and $G_{19}$ each contains a minor isomorphic to $G_{7}$, we discard them.  We also discard $G_{16}$ as it contains a minor isomorphic to $G_{6}.$  Therefore $Q=M(G)$, $G=G_{17}$ or $G$ is a coextension of $G_{17}$ by one or two elements. Observe that any coextension of $G_{17}$  by an edge  contains $G_{7}$ as a minor or contains  a $2$-edge cut.  We discard the former case as in this case $M$ contains a minor isomorphic to $M(G_{7}).$  In the latter case, $Q$ contains a  2-cocircuit, a contradiction to  Lemma \ref{c2nocckts}. Also, a coextension by an element of a graph containing 2-edge cut contains a 2-edge cut. 
	
	Hence $G=G_{17}.$  Therefore $G$ has eight edges and two disjoint 2-circuits.  Since $F_7$ has seven elements,  there is  $T \subseteq E(G)$ with $|T| = 3$ such that $M(G)_T/x \cong F_7$ for some $ x \in T.$  If a 2-circuit of $M(G)$ does not intersect $T$ or is contained in $T,$  then it remains a circuit in $M(G)_T$ by Lemma \ref{c18} and so we get a 2-circuit or  a loop  in $M(G)_T/x,$ a contradiction. Hence  $T$ contains exactly one edge from each 2-circuit of $G.$  We may assume that $x$ does not belong to a 2-circuit of $G.$   Then there is a triangle in $G$ containing exactly two edges from $T$ one of which is $x.$ Hence this 3-circuit of $Q$ is preserved in  $Q_T$ also. Therefore, we get a  2-circuit in $Q_T/x \cong F_7,$ a contradiction.  
			 
	\vskip.3cm
	\noindent 
	{\bf Case (iii). } Suppose $F \cong M^*(K_5)$.  
	
	Since $M$ avoids $M(G_6)$ and $M(G_7)$ as minors, the minor $N/a$ of $M$ also avoids them.	By Lemma \ref{c2mfm*5}, $N/a$  is isomorphic to $M(G_{14})$, where $G_{14}$ is the graph as shown in Figure 3. As $G_{14}=G_{8}$,  $M$ contains a minor isomorphic to $M(G_{8}),$ a  contradiction.
	
	\vskip.3cm
	\noindent 
	{\bf Case (iv).} Suppose  $F  \cong M^*(K_{3,3})$. 
	
	The matroid $N/a$ does not contain a minor $M(G_6)$ or $M(G_7)$. Therefore, by Lemma \ref{c2mfm*33}, $N/a$ is isomorphic to  $M(G_{13})$, where $G_{13}$ is the graphs as shown in Figure 3.   There are three pairs of parallel edges incident to the same vertex in $G_{13}$.  Therefore, if $T$ is a set containing one edge from each pair of parallel edges, then $M(G_{13})_T$ is graphic. However, $Q_T$ is not graphic. Hence $ Q \ncong M(G_{13}).$ Therefore $Q = M(G)$ is a coextension of $M(G_{13})$ by one, two or three elements.

	There are two coextensions of $G_{13}$ by an edge without containing  $G_{6}$ and $G_{7}$ as minors. They are shown in Figure 10 as $G_{20}$ and $G_{21}$. Note that $G_{21}=G_{9}$.   	Hence, if $Q$ arises from  $G_{9},$ then we get a contradiction to our assumption.

	Suppose $Q$ arises from $G_{20}.$  Then $Q$ is isomorphic to $M(G_{20})$ or it is a coextension of $M(G_{20})$ by one or two elements. Suppose $Q$ is isomorphic to $M(G_{20}).$  Then $Q$ is Eulerian and there is a set $T$ of elements of $Q$ with  $|T|=3$ such that $Q_T /x\cong M^*(K_{3,3}) $  for some $x \in T.$ Let $C^*$ be a cocircuit of $Q$ containing $T.$ Then $|C^*|$ is an even integer and so $|C^* - T|$ is an odd integer. By Definition \ref{c11}, both $T$ and $C^*- T$ are cocircuits in $Q_T.$ Hence $C^*-T$ is an odd cocircuit in $Q_T/x.$  This shows that $M^*(K_{3,3})$ contains an odd circuit, a contradiction to the fact that $M^*(K_{3,3})$ is an Eulerian matroid. Hence $Q$ is a coextension of $M(G_{20})$ by one or two elements. However, every coextension of $M(G_{20})$  without containing a 2-edge cut contains a minor isomorphic to $M(G_{7})$, a contradiction.
	
	Hence $M_T$ does not contain $F$ as a minor for any $F \in \mathcal{F}$.  Therefore $M_T$ is graphic. 
\end{proof}

\section{Forbidden Minors for the Splitting by any Set}

In previous sections,  we obtained the forbidden minors for the class of graphic matroids $M$ whose splitting matroids $M_T$ with $|T| = 3$  are graphic. The method developed there can be applied to  the general case.

Let $ k \geq 2$ be an integer. Denote by $ \mathcal{M}_k,$ the class of graphic matorids $M$ whose splitting matroids $M_T,$ for any $ T \subseteq E(M)$ with $|T| = k,$ are graphic.  

Theorem \ref{c2bm} shows that the class $ \mathcal{M}_2$ has three forbidden minors. In Theorem \ref{c2grspl3}, we have obtained forbidden minors for the class $ \mathcal{M}_3$, some of which are single-element extensions of forbidden minors for the class $ \mathcal{M}_2.$ Now, we see that this method of obtaining forbidden minors for the class $\mathcal{M}_3$ can be generalized to get forbidden minors for the class $ \mathcal{M}_k,$ where $k\geq 4.$

Suppose $M$ is a forbidden minor for the class $ \mathcal{M}_{k-1}.$ Then, there is  $T_1 \subseteq E(M)$ with $|T_1| = k-1$ such that $ M_{T_1}$ is not graphic. Let $ \tilde{M}$ be the extension  of $M$ by an element, say $x.$  Let $  T_1 \cup \{x\}=T.$ Then $ T \subseteq E(\tilde{M})$ with $|T| = k.$ Further, $\tilde{M}_T  \backslash x = M_{T_1}.$ Therefore $\tilde{M}_T$ is not graphic. Hence $\tilde{M}$ contains a forbidden minor for the class $\mathcal{M}_{k}.$  Thus, some forbidden minors for $ \mathcal{M}_k$ arise recursively from  the forbidden minors for the  class $ \mathcal{M}_{k-1}$.   

We can find the other type of forbidden minors for $\mathcal{M}_k$ by generalizing the results giving forbidden minors for the class $\mathcal{M}_3$.  Recall that $\mathcal{F}=\{F_7, F^*_7, M^*(K_5), M^*(K_{3,3})\}$.

The following lemma is a generalization of Lemma \ref{c2minle}.

\begin{lemma}   \label{genmin}
	Let $k\geq 4$  and let $M$ be a graphic matroid.  Suppose $M_T$ contains a minor isomorphic to $F$ for some $T \subseteq E(M)$ with $|T| = k$ and $ F \in \mathcal{F}$.  Then $M$ has a minor $N$ containing $T$  such that one of the following holds.
	\begin{enumerate}[label=(\roman*).]
		\item   $N_T\cong F.$
		\item   $N_T/ T'\cong F$ for some non-empty subset $T'$ of $T.$ 
		\item $N$ is isomorphic to a single element extension of one of the forbidden minors of $ \mathcal{M}_{k-1}.$
	\end{enumerate} 
\end{lemma}
\begin{proof}
	The proof is similar to the proof of Lemma \ref{c2minle}. 
\end{proof}

As in  the proof of Lemmas \ref{c2nocckts} and \ref{c2nocirut}, one can easily prove that the minor $N$ satisfying Condition (i) or (ii) of Lemma \ref{genmin} does not contain a coloop and a 2-cocircuit.   The following lemma generalizes Lemma \ref{c2mml}. 	
\begin{lemma} \label{gecammm}
	Let $k \geq 4$ and let $M$ be a graphic matroid.   Suppose $M$ does not contain a coloop and 2-cocircuit and $M_T/T' \cong F$ for some $T\subseteq E(M)$ with $|T|= k$ and $T' \subseteq T$, where $F \in \mathcal{F}$.   Then there is a binary matroid $N$ containing an element    $a$ such that $N \backslash a \cong F.$ Further,  $N/a$ is a minor of $M$ such that either $M = N/a $ or $M$ is a coextension of $N/a$ by at most $k$ elements.
\end{lemma}
\begin{proof}
	The proof is similar to the proof of Lemma \ref{c2mml}. 
\end{proof}

We 	combine Lemmas \ref{c2mff*7}, \ref{c2mff7}, \ref{c2mfm*33} and \ref{c2mfm*5} in the next result. 
\begin{lemma} \label{gc2mfaall}
	Let $N$ be a binary matroid and $a \in E(N)$. If $N \backslash a \cong  F$ for some $F \in \mathcal{F}$ and $N /a$ is the graphic matroid, then $N/a$ is isomorphic to one of the  circuit matroid of graphs $G_{10}$, $G_{11}$, $G_{12}$, $G_{13}$ and $G_{14}$, where $G_{10}$, $G_{11}$, $G_{12}$, $G_{13}$ and $G_{14}$ are the graphs as shown in Figure 3.
\end{lemma}

It follows from Lemmas \ref{gecammm} and \ref{gc2mfaall} that any forbidden minor $P$ for the class $\mathcal{M}_k $ is isomorphic to 

\begin{enumerate}[label=(\roman*)]

\item  a single-element extension of some forbidden minor for the class $\mathcal{M}_{k-1}$ or 

\item  $M(G_i)$ for some $i \in \{10,11,12,13,14\}$ or 

\item  a coextension of $M(G_i)$ by $h$ elements for some $i \in \{10,11,12,13,14\}$ and $1\leq h \leq k$. 
\end{enumerate}


\begin{thebibliography}{99}
\bibitem{borse2014excluded}  {Y. M. Borse, M. M. Shikare and K. V. Dalvi, {Excluded-minors for the class of cographic splitting matroids,} \textit{Ars Combin.} {\bf 115} (2014), 219-237.}

\bibitem{bm1} {Y. M. Borse and  G. Mundhe, Graphic and cographic  $\Gamma$-extensions of binary matroids, \textit{ Discuss. Math. Graph Theory} \textbf{38} (2018), 889-898.}
\bibitem{bm2} {Y. M. Borse and  G. Mundhe, On $n$-connected splitting matroids,  \textit{AKCE Int. J. Graphs Comb.} \textbf{16}(1) (2019), 50-56.}

\bibitem{cw} {R. Chen  and G. Whittle, On recognizing frame and lifted-graphic matroids, \textit{J. Graph Theory} \textbf{87}(1) (2018), 72-76.}

\bibitem{cg} {R. Chen and J. Geelen, Infinitely many excluded minors for frame matroids
and for lifted-graphic matroids, \textit{J. Combin. Theory Ser. B} \textbf{133} (2018), 46-53.}

\bibitem{dfd} {M. DeVos, D. Funk and I. Pivotto,  When does a biased graph come from a group labelling?, \textit{Adv. in Appl. Math.}\textbf{ 61} (2014), 1-18.}

	\bibitem{fd} {D. Funk and D. Mayhew, On excluded minors for classes of graphical matroids, \textit{Discrete Math.} 	\textbf{341}(6) (2018),  1509-1522.}




\bibitem{fleischner1990eulerian} {H. Fleischner,  {\it Eulerian Graphs and Related Topics Part 1, Vol. 1,}  North Holland, Amsterdam, 1990.}

\bibitem{j1} {Jim Geelen, Bert Gerards, and Geoff Whittle, Solving Rota’s conjecture,\textit{ Notices Amer. Math. Soc.} \textbf{61}(7) (2014),736-743.}
\bibitem{j2} {Jim Geelen, Bert Gerards, and Geoff Whittle, The highly connected matroids in minor-closed classes, \textit{Ann. Comb.} \textbf{19}(1) (2015), 107–123.}

\bibitem{harary6graph} { F. Harary, {\it Graph theory}, Narosa Publishing House, New Delhi, 1988.}

\bibitem{mills} {A. Mills, {On the cocircuits of a splitting  matroid,}  \textit{Ars Combin.}  {\bf 89}  (2008), 243-253.}

\bibitem{oxley2006matroid} {J. G. Oxley, {\it Matroid Theory}, Oxford University Press, Oxford, 1992.}

\bibitem{raghunathan1998splitting} { T. T. Raghunathan, M. M. Shikare   and  B. N. Waphare, {Splitting in a binary matroid,}  \textit{Discrete Math}. {\bf 184} (1998), 267-271.}

\bibitem{shikare2011generalized}  {M. M. Shikare, G. Azadi and B. N. Waphare, {Generalized splitting operation for binary matroids and its applications,} \textit{J. Indian Math. Soc. New Ser.} {\bf 78} No.1-4  (2011), 145-154.}

\bibitem{shikare2010excluded} {M. M. Shikare  and  B. N. Waphare, {Excluded-Minors for the class of graphic splitting matroids,} \textit{ Ars Combin.} {\bf 97}(2010), 111-127.}
\bibitem{shikareazadi} {M. M. Shikare and G. Azadi, Determination of the bases of a splitting matroid, \textit{European J. Comb.} \textbf{24} (1) (2003), 45-52.}
\bibitem{azanchiler} {M. M. Shikare, H. Azanchiler and B. N. Waphare, The cocircuits of splitting matroids, \textit{J. Indian Math. Soc.} \textbf{74} No. 3-4 (2007), 185-202.}
\bibitem{zas} {T. Zaslavsky, Biased graphs. II. The three matroids, \textit{J. Combin. Theory Ser. B}
	\textbf{51}(1) (1991), 46-72.}
\end{thebibliography}
\end{document}